\documentclass[11pt,a4paper,reqno]{amsart}
\usepackage{graphicx}
\usepackage{epsfig}
\usepackage{cite}
\usepackage{longtable}
\usepackage{setspace}
 \newtheorem{thm}{Theorem}[subsection]
 \newtheorem{cor}[thm]{Corollary}
 \newtheorem{lem}[thm]{Lemma}
 \newtheorem{prop}[thm]{Proposition}
 \theoremstyle{definition}
 \newtheorem{defn}[thm]{Definition}
 \theoremstyle{remark}
 
 \newtheorem{pro}[thm]{Problem}
 \newtheorem{con}[thm]{Conjecture}

\begin{document}
\setcounter{page}{1}
\begin{flushleft}
\end{flushleft}
\bigskip
\bigskip
\title[Jianxin Wei, Yujun Yang: On the Fibonacci $(p,r)$-cubes] {On the Fibonacci $(p,r)$-cubes}
\author[Jianxin Wei, Yujun Yang: On the Fibonacci $(p,r)$-cubes]{Jianxin Wei,$^1$ Yujun Yang$^2$}
\thanks{$^1$School of Mathematics and Statistics Science, Ludong University, Yantai, 264025 P.R. China
\\ \indent\,\,\,e-mail: wjx0426@ldu.edu.cn
\\ \indent $^2$School of Mathematics and Information Sciences, Yantai University, Yantai, 264005 P.R. China
\\ \indent\,\,\,e-mail: yangyj@yahoo.com
\\ \indent}

\begin{abstract}
In this paper,
first it is shown that the `Fibonacci $(p,r)$-cube'(denoted as $I\Gamma_{n}^{(p,r)}$) studied in many papers is a new topology for interconnection network different from the original one (denoted as $O\Gamma_{n}^{(p,r)}$) presented by Egiazarian and Astola.
Then some topological properties of $I\Gamma_{n}^{(p,r)}$ and $O\Gamma_{n}^{(p,r)}$ are studied,
including the recursive structure,
some distance invariants,
and the maximum and minimum degree of them.
Finally, several problems and conjectures on $I\Gamma_{n}^{(p,r)}$ and $O\Gamma_{n}^{(p,r)}$ are listed.


\end{abstract}
\maketitle

\smallskip
\section{Introduction}
Let $B = \{0,1\}$ and $\mathcal{B}_{n}=\{b_{1}b_{2}\ldots b_{n}|b_{i}\in B, i=1,\ldots, n\}$.
An element of $\mathcal{B}_{n}$ is called a binary word (or simply a word) of length $d$.
We use the product notation for words meaning concatenation,
for example, $1^{s}$ is the word $11\ldots1$ of length $s$.
A word $f$ is called a \emph{factor} of a word $\mu$ if $f$ appears as a sequence of $|f|$ consecutive coordinates of $\mu$.
A non-extendable sequence of contiguous equal digits in a word $f$ is called a \emph{block} of $f$,
for example, $1^{2}0^{3}1^{4}$ has three blocks.

The $n$-dimensional hypercube $Q_{n}$ is the graph defined on the vertex set $\mathcal{B}_{n}$,
two vertices $\alpha=a_{1}a_{2}\ldots a_{n}$,
$\beta=b_{1}b_{2}\ldots b_{n}\in \mathcal{B}_{n}$ being adjacent if $a_{i}$ $\neq$ $b_{i}$ holds for exactly one $i$ $\in$ $\{1,\ldots, n\}$.
Inspired by classical Fibonacci sequence of numbers,
Hsu \cite{Hsu} introduced the Fibonacci cube $\Gamma_{n}$ for $n \geq 1$.
Let $\mathcal{F}_{n}= \{ b_{1}b_{2} \ldots b_{n}|$ $ b_{i} \in B, b_{i}b_{i+1} = 0, i=1, \ldots, n-1\}$.
Then $\Gamma_{n}$ has $\mathcal{F}_{n}$ as the vertex set,
two vertices being adjacent if they differ in exactly one coordinate.
The cube $\Gamma_{5}$ is shown in Figure. 1.

\hspace{1mm}
\begin{center}
\includegraphics[scale=0.70]{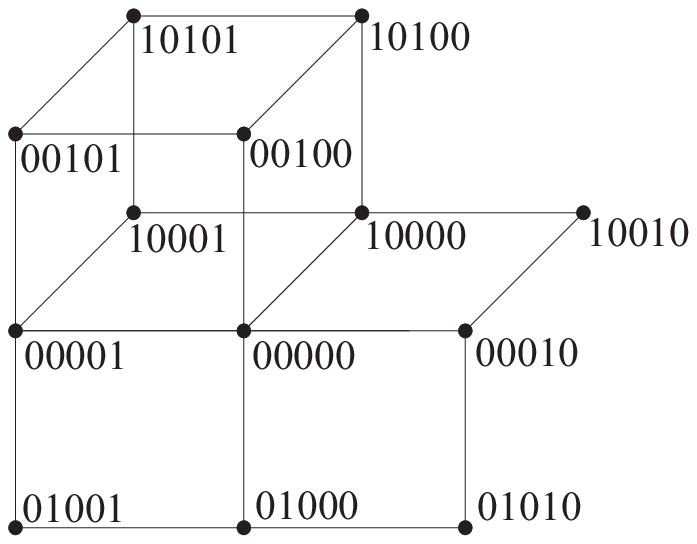}\\
{\footnotesize Figure 1. The 5-dimensional Fibonacci cube $\Gamma_{5}$.}
\end{center}
\hspace{3mm}

Some structural and enumerative properties of the Fibonacci cubes were given in \cite{MS}.
Fibonacci cubes have already found some applications in chemical graph theory.
For example,
it is showed that Fibonacci cubes are precisely the resonance graphs of zigzag hexagonal chains \cite{KZ}.
For more results on application and structure of  Fibonacci cubes, see \cite{K1} for a survey.

When Fibonacci cubes was introduced,
it soon became increasingly popular.
Numerous variants and generalizations of Fibonacci cubes
are proposed and investigated \cite{WY,IKR,KM,AV,JW,QW,W1,WYZ,AZS,HCD}.
In the present paper,
we pay a special attention to the Fibonacci $(p,r)$-cubes introduced by Egiazarian and Astola \cite{Egiazarian}.

The Fibonacci $(p,r)$-numbering system is defined by the following recurrence

\begin{align}\label{1.1}
\varphi_{(p,r)}(i)=\sum_{j=0}^{r}\varphi_{(p,r)}(i-pj-1), i>1,
\end{align}

\noindent
with initial values $\varphi_{(p,r)}(i)=0$ for $i<0$ and $\varphi_{(p,r)}(i)=1$ for $i=0$,
where $p\geq1$ and $r\geq1$.
Let $K$ be an positive integer such that $K<\varphi_{(p,r)}(n)$ for some $n>0$ .
Then the $O$-Fibonacci $(p,r)$-code of $K$ is defined uniquely as $b_{n-1}\ldots b_{p}$ such that $K=\sum_{i=p}^{n-1}b_{i}\varphi_{(p,r)}(i)$ and the following hold:

(1) if $b_{i}=1$ then $b_{i-1}=b_{i-2}=b_{i-p+1}=0$,
i.e. there is at least $(p-1)$ 0s between two 1s (which is called `consecutive' 1s) in $b_{n-1}\ldots b_{p}$;

(2) there are no more than $r$ `consecutive' 1s in $b_{n-1}\ldots b_{p}$.

Some examples of Fibonacci $(p,r)$-numbers are shown in Table 1.
{\renewcommand\baselinestretch{1.2}\selectfont
\begin{center} \fontsize{9.0pt}{10.5pt}\selectfont
\begin{longtable}{c|c|c|c|c|c|c|c|c|c|c|c|c|c}
\caption{Examples of systems of Fibonacci $(p,r)$-numbers.}\\
\hline
$i$&0&1&2&3&4&5&6&7&8&9&10&11&12\\
\hline
$\varphi_{(1,1)}(i)$&1&1&2&3&5&8&13&21&34&55&89&144&233\\
\hline
$\varphi_{(1,3)}(i)$&1&1&2&4&8&15&29&56&108&208&401&773&1490\\
\hline
$\varphi_{(2,1)}(i)$&1&1&1&2&3&4&6&9&13&19&28&41&60\\
\hline
$\varphi_{(2,2)}(i)$&1&1&1&2&3&5&8&12&19&30&47&85&116\\
\hline
\end{longtable}
\end{center}
\par}

\begin{defn}\label{defn1}\cite{Egiazarian}
\emph{For positive integers $n$, $p$ and $r$,
the Fibonacci $(p,r)$-cube of size $\varphi_{(p,r)}(n)$,
denoted by $G_{(p,r;n)}=(V,E)$,
where $V=\{0,1,\ldots,\varphi_{(p,r)}(n)-1\}$ and the edge $(i,j)\in E$ if and only if the Fibonacci $(p,r)$-codes of $i$ and $j$ differ exactly in one coordinate.}
\end{defn}

The graphs $G_{(1,3;4)}$, $G_{(2,1;8)}$ and $G_{(2,2;7)}$ are shown in Figure 2$(a),(b)$ and $(c)$ respectively,
where each vertex is labelled with the Fibonacci $(p,r)$-code and the corresponding number.
For example,
$10100(11)$ in graph $G_{(2,2;7)}$ means the vertex $11$ with its $O$-Fibonacci $(2,2)$-code $10100$.

Since the length of the $O$-Fibonacci $(p,r)$-code of the vertices of $G_{(p,r;n)}$ is $n-p$,
the graph $G_{(p,r;n)}$ is denoted as $O\Gamma_{n-p}^{(p,r)}$ for convenience,
where `$O$' means the `original' Fibonacci $(p,r)$-cube to distinguish the `imitative' one introduced following.
Note that the symbol $O\Gamma_{n-p}^{(p,r)}$ is more flexible in the definition of an iteration,
see Section 2.

In papers \cite{K1,KR,OZ,OZY,JZ},
the ``Fibonacci $(p,r)$-cube'' is studied from several different aspects.
Although it also was called ``Fibonacci $(p,r)$-cube'' and pointed out to be the cubes defined in \cite{Egiazarian},
we find that the cube in these papers is a new topological structure differing from the existing Fibonacci like-cubes.
We call those cubes the imitative Fibonacci $(p,r)$-cube and denoted as $I\Gamma_{n}^{(p,r)}$,
which is defined as follows.

Let $n,$ $p$ and $r$ be positive integers.
Then a $I$-Fibonacci $(p,r)$-code of length $n$ is a word of length $n$
in which there are at most $r$ consecutive 1s and at least $p$ element $0$s between two sub-words composed of (at most $r$) consecutive 1s.

\begin{defn}\label{defn2}\cite{OZY}
\emph{Let $\mathcal{F}_{n}^{(p,  r)}$ denote the set of $I$-Fibonacci $(p,r)$-codes of length $n$.
Then $I\Gamma_{n}^{(p,r)}=(V,E)$ is the graph such that $V=\mathcal{F}_{n}^{(p,r)}$ and two vertices are adjacent if they differ in exactly one coordinate.}
\end{defn}

For example
let $(p,r)=(2,2)$.
Then $I\Gamma_{5}^{(2,2)}$ is the cube shown in Figure 2$(d)$.
We think the main difference between the definitions of $O\Gamma_{n}^{(p,r)}$ and $I\Gamma_{n}^{(p,r)}$ is the meaning of the `$r$ consecutive 1s'.
In the codes of the vertex of $O\Gamma_{n}^{(p,r)}$ it means that $(10^{p-1})^{r}$,
but in $I\Gamma_{n}^{(p,r)}$ it means that $1^{r}$.
Hence, we claim that $O\Gamma_{n}^{(p,r)}$ and $I\Gamma_{n}^{(p,r)}$ are not the same one.
The following study in the present paper also confirmed this point.

\hspace{1mm}
\begin{center}
\includegraphics[scale=0.75]{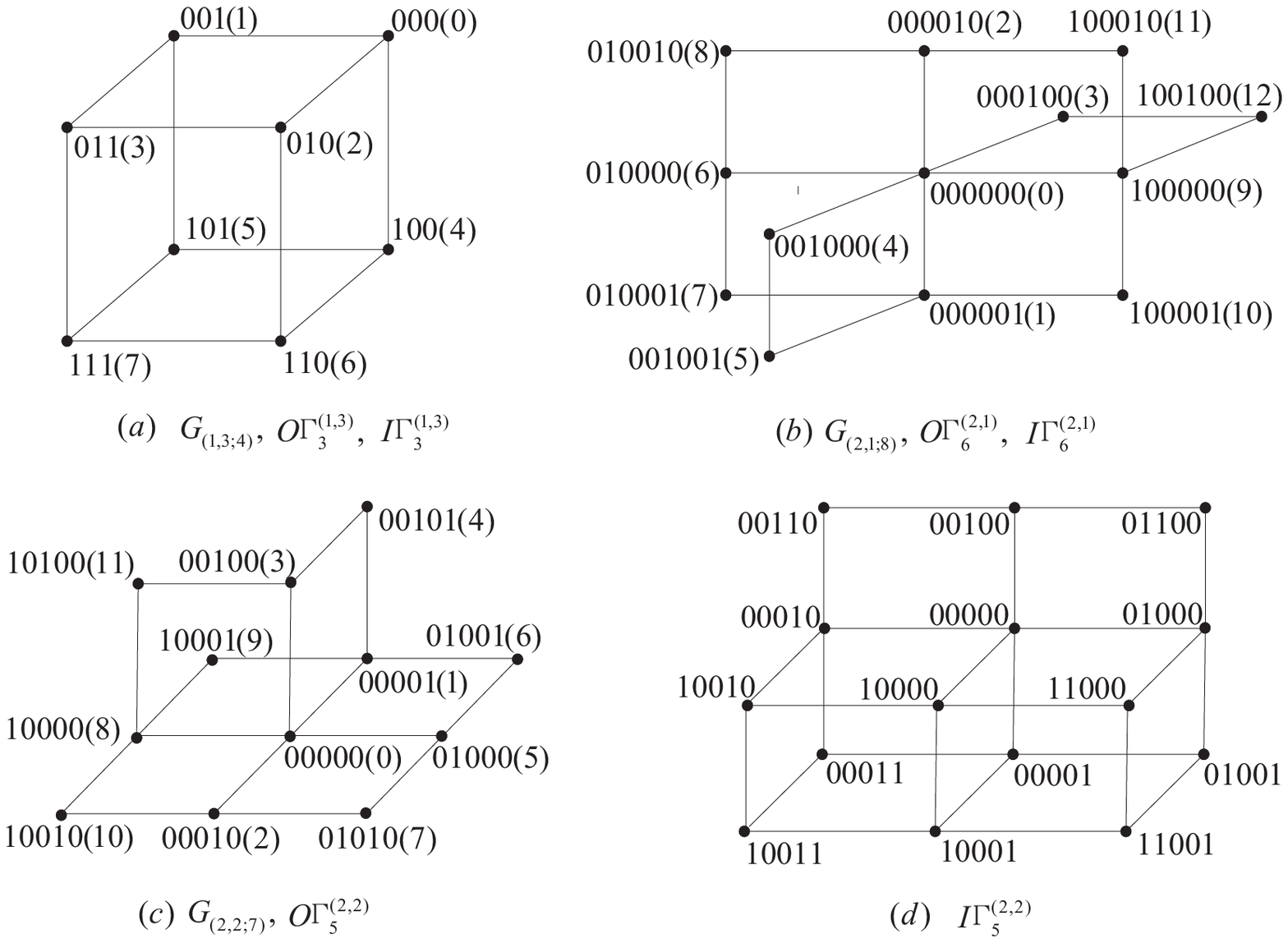}\\
{\footnotesize Figure 2. Some examples of $G_{(p,r;n)}$, $O\Gamma_{n}^{(p,r)}$ and $I\Gamma_{n}^{(p,r)}$.}
\end{center}
\hspace{3mm}

The rest of paper is organized as follows.
In the next section the recursive structure of $O\Gamma_{n}^{(p,r)}$ and $I\Gamma_{n}^{(p,r)}$ are studied,
and some preliminary observations and properties are given.
Then, in Section 3,
some distance invariants of $O\Gamma_{n}^{(p,r)}$ and $I\Gamma_{n}^{(p,r)}$ are studied,
including the diameter, the radius and the center of these cubes.
In Section 4,
the maximum and minimum degree of $O\Gamma_{n}^{(p,r)}$ and $I\Gamma_{n}^{(p,r)}$ are determined.
All the results obtained in these sections further shown that $O\Gamma_{n}^{(p,r)}$ and $I\Gamma_{n}^{(p,r)}$ are not the same cubes in general.
In the last section,
several problems and conjectures for further research on these two cubes are listed.

\smallskip
\section{Recursive structure}

In this section,
the recursive structure of the vertex set and the edge set of $O\Gamma_{n}^{(p,r)}$ and $I\Gamma_{n}^{(p,r)}$ are given,
respectively.
Based on these structure,
some general properties of these cubes are shown,
and the relationships between $O\Gamma_{n}^{(p,r)}$ and $I\Gamma_{n}^{(p,r)}$ is further revealed.

It is useful to define $V(Q_{0})=V(\Gamma_{0})=V(O\Gamma_{0}^{(p,r)})=V(I\Gamma_{0}^{(p,r)})=\{\lambda\}$,
where $\lambda$ denotes the empty word.
Let $b$ be a word and $\mathcal{A}\subseteq \mathcal{B}_{n}$.
Then $b\mathcal{A}$ is the set of word obtained from $\mathcal{A}$ by appending $b$ in front of each of the terms of $\mathcal{A}$,
that is,
$b\mathcal{A}=\{ba|a\in\mathcal{A}\}$.

\subsection{Recursive structure of the vertex set}
It is well known that the vertex set $\mathcal{B}_{n}$ of hypercube $Q_{n}$ satisfies with

\begin{align}\label{2.1}
\mathcal{B}_{n}=0\mathcal{B}_{n-1}\cup 1\mathcal{B}_{n-1}\quad, n\geq1,
\quad and \quad \mathcal{B}_{0}=\{\lambda\}
\end{align}

\noindent
So the order of $Q_{n}$ is $2^{n}$ for $n\geq0$ by Eq. (\ref{2.1}).
The vertex set $\mathcal{F}_{n}$ of Fibonacci cube $\Gamma_{n}$ satisfies with

\begin{align}\label{2.2}
\mathcal{F}_{n}=0\mathcal{F}_{n-1}\cup 10\mathcal{F}_{n-2}, n\geq2,
\quad and \quad \mathcal{F}_{0}=\{\lambda\}, \mathcal{F}_{1}=\{0,1\}.
\end{align}

\noindent
By Eq. (\ref{2.2}),
the order of Fibonacci cube $\Gamma_{n}$ satisfies that $|\mathcal{F}_{n}|=|\mathcal{F}_{n-1}|+|\mathcal{F}_{n-2}|$ for $n\geq2$ with $|\mathcal{F}_{0}|=1$, $|\mathcal{F}_{1}|=2$.
Let $p\geq1, r\geq1$ and $n\geq1$.
Then it \cite{Egiazarian} is shown that the vertex set of $O\Gamma_{n}^{(p,r)}$ can be defined recursively by

\begin{align}\label{2.3}
V(O\Gamma_{n}^{(p,r)})=0V(O\Gamma_{n-1}^{(p,r)})\cup 10^{p-1}0V(O\Gamma_{n-p-1}^{(p,r)})\cup\ldots \cup (10^{p-1})^{r}0V(O\Gamma_{n-pr-1}^{(p,r)}),
\end{align}

\noindent
with the initial conditions $V(O\Gamma_{n}^{(p,r)})=\emptyset$ for $n<0$,
$V(O\Gamma_{n}^{(p,r)})=\{\lambda\}$ for $n=0$.
According to Eq. (\ref{2.3}),
the following result holds.

\begin{prop}\label{prop2.1}
Let $n$ and $p$ be given integers.
Then $O\Gamma_{n}^{(p,r)}\cong O\Gamma_{n}^{(p,r')}$ for all $r$ and $r'$ such that $n\leq pr$ and $n\leq pr'$.
\end{prop}

By Eq. (\ref{2.3}),
the iterative formula of the order of $O\Gamma_{n}^{(p,r)}$, consist with Eq. (\ref{1.1}), is

\begin{align}\label{2.4}
|V(O\Gamma_{n}^{(p,r)})|=|V(O\Gamma_{n-1}^{(p,r)})|+|V(O\Gamma_{n-p-1}^{(p,r)})|+\ldots+|V(O\Gamma_{n-pr-1}^{(p,r)})|.
\end{align}

By the definition of the cube $I\Gamma_{n}^{(p,r)}$,
the following property holds obviously.

\begin{prop}\label{prop2.2}
Let $n$ and $p$ be given integers.
Then the vertex set of $I\Gamma_{n}^{(p,r)}$ satisfies with

\begin{align}\label{2.5}
V(I\Gamma_{n}^{(p,r)})=0V(I\Gamma_{n-1}^{(p,r)})\cup 10^{p}V(I\Gamma_{n-p-1}^{(p,r)})\cup\ldots \cup 1^{r}0^{p}V(I\Gamma_{n-p-r}^{(p,r)}),
\end{align}

\noindent
with the initial conditions $V(I\Gamma_{n}^{(p,r)})=\emptyset$ for $n<0$,
$V(I\Gamma_{n}^{(p,r)})=\{\lambda\}$ for $n=0$.
\end{prop}

By Proposition \ref{prop2.2},
the iterative formula of the order of $I\Gamma_{n}^{(p,r)}$ satisfies:

\begin{align}\label{2.6}
|V(I\Gamma_{n}^{(p,r)})|=|V(I\Gamma_{n-1}^{(p,r))}|+|V(I\Gamma_{n-p-1}^{(p,r)})|+\ldots+|V(I\Gamma_{n-p-r}^{(p,r)})|.
\end{align}

Postal Network \cite{WY} $PN_{p}(n)$ of dimension $n$ with a parameter $p\geq1$ is a subgraph of $Q_{n}$ induced on vertices

\begin{align}\label{2.7}
PN_{p}(n)=0PN_{p}(n-1)\cup 10^{p-1}PN_{p}(n-p) \quad &if\quad n>p,
\end{align}

\noindent
and $PN_{p}(n)=\{\nu\in Q_{n}|w(\nu)\leq1\}$ if $n\leq p$,
where $w(\nu)$ is the number of $1$s in $\nu$.

Eqs. (\ref{2.3}-\ref{2.6}) show that the vertex sets of $O\Gamma_{n}^{(p,r)}$ and $I\Gamma_{n}^{(p,r)}$ are different in general,
but they are the same one for some special cases.
By Eqs. (\ref{2.1}-\ref{2.7}),
it is easy to see the following holds.

\begin{prop}\label{prop2.3}
Let $p\geq1$ and $n\geq1$.
Then $O\Gamma_{n}^{(1,n)}\cong I\Gamma_{n}^{(1,n)}\cong Q_{n}$,
$O\Gamma_{n}^{(1,1)}\cong I\Gamma_{n}^{(1,1)}\cong \Gamma_{n}$, and
$O\Gamma_{n}^{(p,1)}\cong I\Gamma_{n}^{(p,1)}\cong PN_{p+1}(n)$.
\end{prop}

By Proposition \ref{prop2.3},
both $O\Gamma_{n}^{(p,r)}$ and $ I\Gamma_{n}^{(p,r)}$ unify a wide range of topologies such as hypercube,
classical Fibonacci cube and postal network.

By Eqs. (\ref{2.4}) and (\ref{2.6}),
the speed growth of the order of $V(O\Gamma_{n}^{(p,r)})$ is far greater than $V(I\Gamma_{n}^{(p,r)})$,
although $|V(O\Gamma_{n}^{(p,r)})|$ and $|V(I\Gamma_{n}^{(p,r)})|$ have the same initial conditions.
The following proposition is also obvious.

\begin{prop}\label{prop2.4}
Suppose $p>1$ and $r>1$.
Then $|V(O\Gamma_{n}^{(p,r)})|=|V(I\Gamma_{n}^{(p,r)})|$ for $n\leq p+1$,
and $|V(O\Gamma_{n}^{(p,r)})|<|V(I\Gamma_{n}^{(p,r)})|$ for $n> p+1$.
\end{prop}

Many Fibonacci like-cubes can be obtained from hypercube by some word forbidden to appear in the codes of hypercube.
Such as Fibonacci cubes, Lucas cubes, and Daisy cubes \cite{KM}.
From the point of view,
the following property holds:

\begin{prop}\label{prop2.5}
If $p=1$,
then both $O\Gamma_{n}^{(p,r)}$ and $I\Gamma_{n}^{(p,r)}$ can be obtained from $Q_{n}$ by removing all vertices that contain the word $1^{r+1}$ as factor.
If $p\geq2$, then $I\Gamma_{n}^{(p,r)}$ can be obtained from $Q_{n}$ by removing all vertices that contain the words
$1^{r+1}$ or $10^{s}1$ for all $s\leq p-1$;
and the cube $O\Gamma_{n}^{(p,r)}$ can be obtained from $Q_{n}$ by removing all vertices that contain the words
$(10^{p-1})^{r}1$ or $10^{s}1$ for all $s\leq p-2$.
\end{prop}

Let $\upsilon=v_{1}v_{2}\ldots v_{n}$.
Then $\upsilon^{R}=v_{n}\ldots v_{2}v_{1}$ is called the reverse word of $\upsilon$.
By Proposition \ref{prop2.5},
the following result holds obviously.

\begin{prop}\label{prop2.6}
Let $v$ be a vertex of $O\Gamma_{n}^{(p,r)}$ (or $I\Gamma_{n}^{(p,r)}$). Then $v^{R}$ is a vertex of $O\Gamma_{n}^{(p,r)}$ (or $I\Gamma_{n}^{(p,r)}$) too.
\end{prop}

By Propositions \ref{prop2.4} and \ref{prop2.5},
the following holds obviously.

\begin{prop}\label{prop2.7}
Let $n$, $p$ and $r$ be positive integers.
Then $O\Gamma_{n}^{(p,r)}\cong I\Gamma_{n}^{(p,r)}$ if and only if $p=1$ or $r=1$.
\end{prop}

This result does not mean that there is not the same cube defined by Definitions \ref{defn1} and \ref{defn2}.
For example,
the cubes $O\Gamma_{4}^{(2,2)}$ and $I\Gamma_{4}^{(3,2)}$ are shown in Figure 3$(a)$ and 3$(b)$, respectively.
It is easy to see that $O\Gamma_{4}^{(2,2)}\cong I\Gamma_{4}^{(3,2)}$.

\hspace{1mm}
\begin{center}
\includegraphics[scale=0.75]{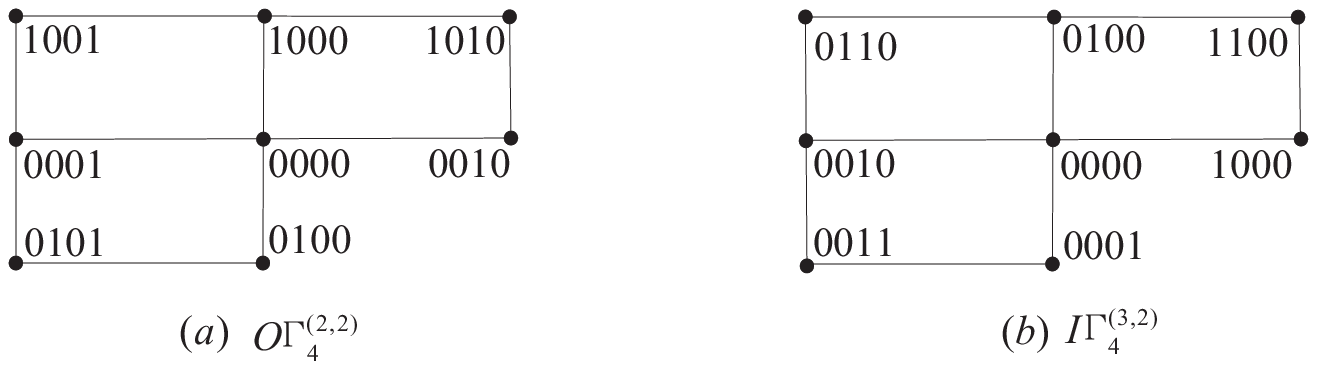}\\
{\footnotesize Figure 3. The graphs $O\Gamma_{4}^{(2,2)}$ and $I\Gamma_{4}^{(3,2)}$.} \end{center} \hspace{3mm}

\subsection{Recursive structure of edge set}

Although Egiazarian and Astola \cite{Egiazarian} gave the iterative formula of the size of $O\Gamma_{n}^{(p,r)}$,
we find it was wrong and correct it here.

\begin{prop}\label{prop2.8}
Let $p\geq1$, $r\geq1$ and $n> pr$.
Then $|E(O\Gamma_{n}^{(p,r)})|=\sum^{r}_{t=0}(|E(O\Gamma_{n-tp-1}^{(p,r)})|+t|V(O\Gamma_{n-tp-1}^{(p,r)})|)$.
\end{prop}

\begin{proof}
By the iterative formula of the vertex set of $O\Gamma_{n}^{(p,r)}$ as shown in Eq. (\ref{2.3}),
the graph $O\Gamma_{n}^{(p,r)}$ can be decomposed into $r+1$ disjoint subgraphs isomorphic
to $O\Gamma_{n-tp-1}^{(p,r)}$ for $t=0,1,\ldots,r$, respectively.
As shown in Figure 4,
there are $|V(O\Gamma_{n-ip-1}^{(p,r)})|$ edges connecting subgraphs $O\Gamma_{n-ip-1}^{(p,r)}$ and $O\Gamma_{n-jp-1}^{(p,r)}$,
$0\leq i< j\leq r$.
So there are $\sum^{r}_{t=0}(t|V(I\Gamma_{n-tp-1}^{(p,r)})|)$ edges between these $r+1$ subgraphs.
This completes the proof.
\end{proof}

\hspace{1mm}
\begin{center}
\includegraphics[scale=0.74]{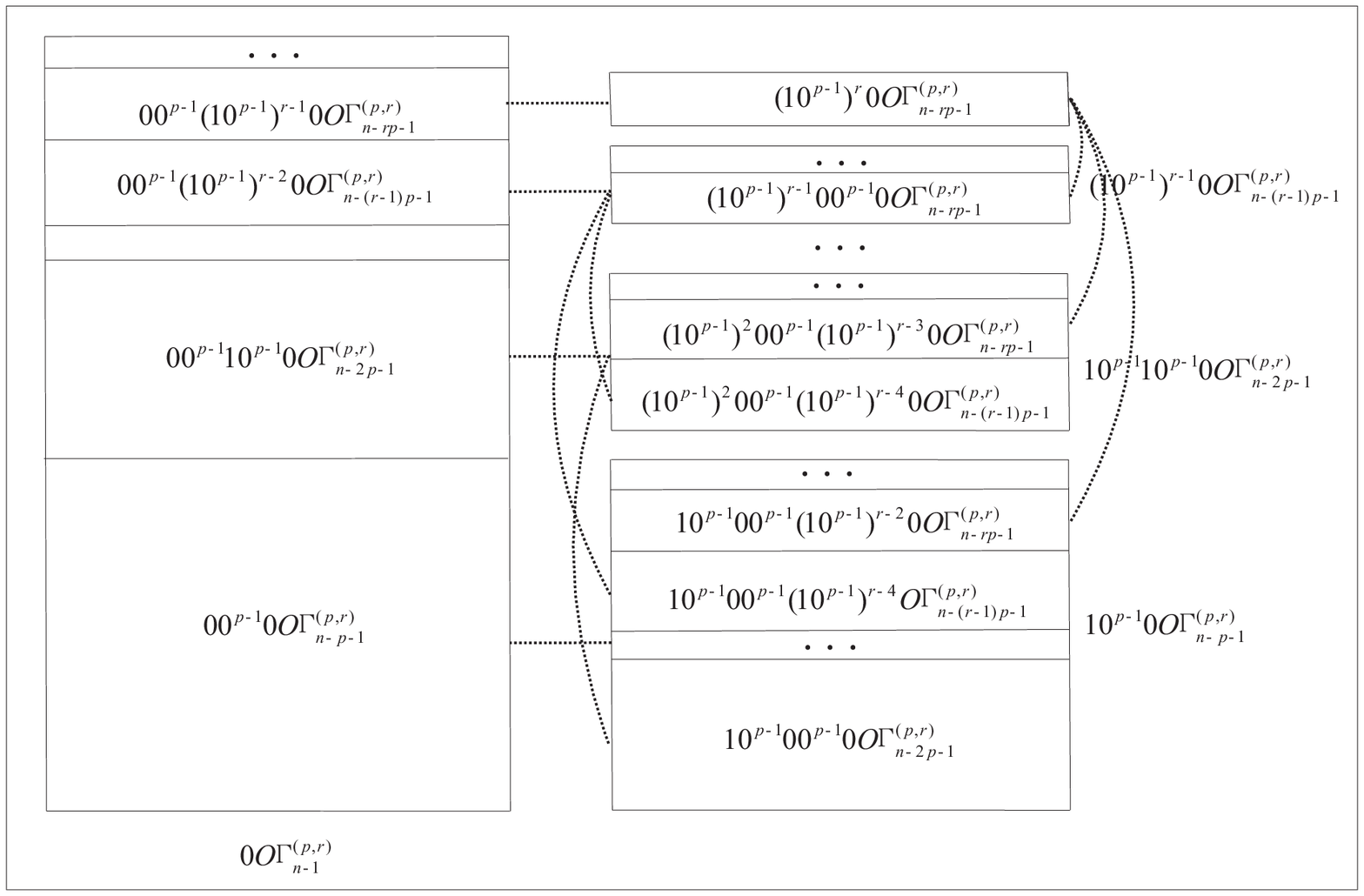}\\
{\footnotesize Figure 4. Illustration of the decomposed graphs of $O\Gamma_{n}^{(p,r)}$.} \end{center}\hspace{3mm}

For the cube $I\Gamma_{n}^{(p,r)}$,
we have the following result.

\begin{prop}\label{prop2.9}
Let $p\geq1,r\geq1$ and $n>p+r$.
Then the following recursive structure holds on the edge set of $O\Gamma_{n}^{(p,r)}$:

$|E(I\Gamma_{n}^{(p,r)})|=|E(I\Gamma_{n-1}^{(p,r)})|+
\sum^{r}_{t=1}(|E(I\Gamma_{n-p-t}^{(p,r)})|+2|V(I\Gamma_{n-p-t}^{(p,r)})|)
-|V(I\Gamma_{n-p-1}^{(p,r)})|$.
\end{prop}

\hspace{1mm}
\begin{center}
\includegraphics[scale=0.70]{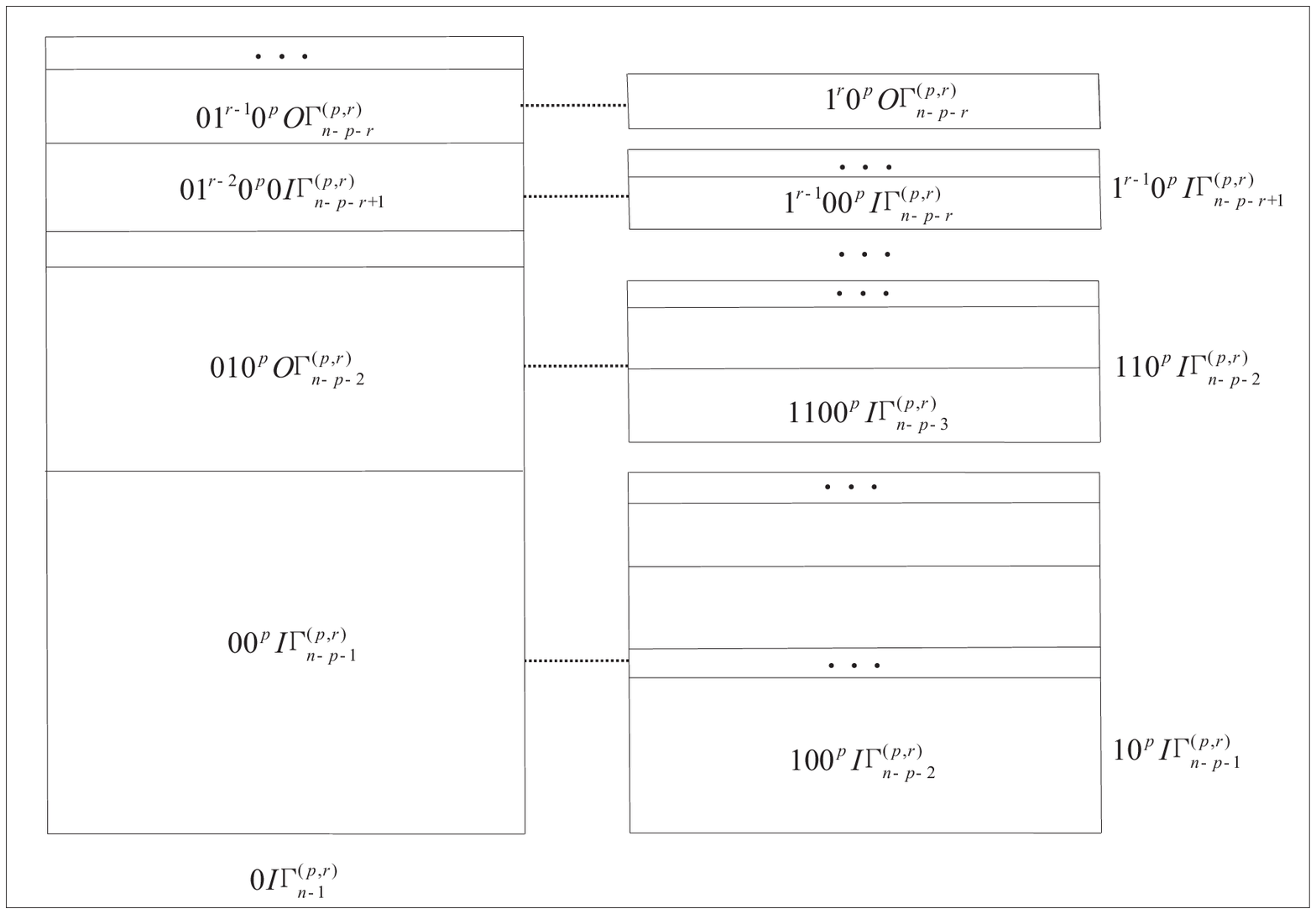}\\
{\footnotesize Figure 5. The decomposed graph of $I\Gamma_{n}^{(p,r)}$.}
\end{center}
\hspace{3mm}

\begin{proof}
As shown in Figure 5,
the graph $I\Gamma_{n}^{(p,r)}$ can be decomposed into $r+1$ disjoint subgraphs isomorphic to $I\Gamma_{n-1}^{(p,r)}$,
$I\Gamma_{n-p-t}^{(p,r)}$ for $t=1,\ldots,r$, respectively,
by the the iterative formula of $I\Gamma_{n}^{(p,r)}$ shown in Eq. (\ref{2.5}).
As there is $|V(I\Gamma_{n-p-t}^{(p,r)})|$ edges between the subgraphs $I\Gamma_{n-p-t}^{(p,r)}$ and $I\Gamma_{n-1}^{(p,r)}$ for $t=1,\ldots,r$; and
$|V(I\Gamma_{n-p-t}^{(p,r)})|$ edges between the subgraphs $I\Gamma_{n-p-t}^{(p,r)}$ and $I\Gamma_{n-p-t+1}^{(p,r)}$, for $t=2,\ldots,r$.
Therefore there are $2\sum^{r}_{t=1}|V(I\Gamma_{n-p-t}^{(p,r)})|-|V(I\Gamma_{n-p-1}^{(p,r)})|$ edges between these $r+1$ subgraphs.
The proof is completed.
\end{proof}

\smallskip
\section{Some distance invariants}

Recall that the eccentricity $e(v)$ of a vertex $v$ in a graph $G$ is the maximum distance between $v$ and the other vertex, i.e.,
$e(v)=\max\limits_{u\in V(G)} d(v,u)$.
The diameter $diam(G)$ of $G$ is the maximal eccentricity $e(v)$ when $v$ runs in $G$, i.e.,
$diam(G)=\max\limits_{v\in V(G)} e(v)=\max\limits_{u,v\in V(G)} d(u,v)$.
The radius $rad(G)$ of $G$ is the minimum eccentricity of the vertices of $G$, i.e.,
$rad(G)=\min\limits_{v\in V(G)} e(v)$.
A vertex is called central if $e(v)=rad(G)$.
The center $\mathcal{Z}(G)$ of $G$ is the set of all central vertices.
Note that the eccentricity, the radius and diameter of Fibonacci cube were studied in \cite{AM} and \cite{MS}.
In this section,
we study the radius,
center and diameter of $O\Gamma_{n}^{(p,r)}$ and $I\Gamma_{n}^{(p,r)}$.

\subsection{Center and Radius}

Let $\mathcal{G}=\{\eta=g_{1}g_{2}\ldots g_{n}|\eta\in\mathcal{B}^{n}\}$
and $w(\eta)=\Sigma_{i=1}^{n}g_{i}$ be the weight of $\eta$ (in other words, $w(\eta)$ be the numbers of $1$s in $\eta$).
Suppose $\eta\in\mathcal{G}$,
$w(\eta)\geq1$ and $g_{i_{1}}=1,\ldots,g_{i_{j}}=1$ for some $j\geq1$.
Then we use $\eta^{\vee}_{i_{t}}$ to denote the word obtained from $\eta$ by inserting $0^{r+1}$ close behind the
$i_{t}$th coordinate, where $t=1,\ldots,j$.
For convenience,
let $\eta^{\vee}_{0}$ be the word $0^{r+1}\eta$.
Then we use $\mathcal{G}^{\vee}$ to denote the set of words of length $n+r+1$:

\begin{align}
\mathcal{G}^{\vee}=\{\eta^{\vee}_{0}|\eta\in \mathcal{G}\}\cup\{\eta^{\vee}_{i_{t}}|\eta=g_{1}g_{2}\ldots g_{n}\in\mathcal{G}~and~g_{i_{1}}=\ldots=g_{i_{j}}=1~for~some~j\geq1\}.
\end{align}

In this subsection,
the radius and center of $O\Gamma_{n}^{(p,r)}$ is determined completely.
For all $p\geq1,r\geq1$ and $n\geq1$,
the radius of the cubes $O\Gamma_{n}^{(p,r)}$ can be expressed by a unified formula,
but the construction of $\mathcal{Z}(O\Gamma_{n}^{(p,r)})$ for the case $p=1$ is a little more complicated than for the case $p\geq2$.
Hence,
we give Theorems \ref{thm4.1} and \ref{thm4.2} to distinguish the case $p=1$ and $p\geq2$.

\begin{thm}\label{thm4.1}
Let $r\geq1$ and $n\geq1$.
Then $0^{n}\in\mathcal{Z}(O\Gamma_{n}^{(1,r)})$, $rad(O\Gamma_{n}^{(1,r)})=\lceil \frac{nr}{r+1}\rceil$, and further
$\mathcal{Z}(O\Gamma_{n}^{(1,r)})$ and $|\mathcal{Z}(O\Gamma_{n}^{(1,r)})|$ can be determined as follows:

$(1)$  if $mod(n,r+1)=0$,
then $\mathcal{Z}(O\Gamma_{n}^{(1,r)})=\{0^{n}\}$ and $|\mathcal{Z}(O\Gamma_{n}^{(1,r)})|=1$;

$(2)$  if $mod(n,r+1)=k\neq0$, then
$\mathcal{Z}(O\Gamma_{n}^{(1,r)})=\mathcal{Z}(O\Gamma_{n-r-1}^{(1,r)})^{\vee}$,
with the initial set $\mathcal{Z}(O\Gamma_{k}^{(1,r)})=V(Q_{k})$
and $|\mathcal{Z}(O\Gamma_{n}^{(1,r)})|=\sum\limits_{j=1}^{k+1}a_{j}\lceil\frac{n}{r+1}\rceil^{k+1-j}$,
where

\begin{center}
$\left[
 \begin{matrix}
   a_{1}\\
   a_{2}\\
   \vdots\\
   a_{k+1}\\
  \end{matrix}
  \right]=$
$\left[
 \begin{matrix}
   1 & 1 & \cdots & 1\\
   2^{k} & 2^{k-1} & \cdots & 1\\
   \vdots & \vdots & \ddots & \vdots \\
   (k+1)^{k} & (k+1)^{k-1} & \cdots &1 \\
  \end{matrix}
  \right]^{-1}$
  $\left[
 \begin{matrix}
   |\mathcal{Z}(O\Gamma_{k}^{(1,r)})|\\
   |\mathcal{Z}(O\Gamma_{k+r+1}^{(1,r)})|\\
   \vdots\\
   |\mathcal{Z}(O\Gamma_{k+k(r+1)}^{(1,r)})|\\
  \end{matrix}
  \right].$
  \end{center}
\end{thm}

\begin{proof}
Let $r\geq1$ and $n=s(r+1)+k$ for some $s\geq 0$ and $0\leq k\leq r$.
Now we show $0^{n}\in\mathcal{Z}(O\Gamma_{n}^{(1,r)})$ and $rad(O\Gamma_{n}^{(1,r)})=\lceil \frac{nr}{r+1}\rceil$.
This need to show that $e(0^{n})=\lceil\frac{nr}{r+1}\rceil$,
and for every vertex $u\neq0^{n}$, $e(u)\geq e(0^{n})$ holds.

First we show that $e(0^{n})=\lceil\frac{nr}{r+1}\rceil$.
It is easy to see that the distance $d(0^{n},v)$ is $w(v)$ for any vertex $v$ of $O\Gamma_{n}^{(1,r)}$.
As in every $r+1$ consecutive elements of a vertex $v$ of $O\Gamma_{n}^{(1,r)}$
there are at least one $0$,
the vertex $v=1^{r}01^{r}0\ldots$ (or some factor of $v$ from the first coordinate of it) is a vertex with the most elements $1$.
Hence, $e(0^{n})=d(0^{n},v)=\lceil\frac{nr}{r+1}\rceil$.
The following sets of numbers are needed to our proof.
Let
\begin{equation} \label{4.1}
\begin{split}
&[n]=\{1,2,\ldots,n\},\\
&\mathcal{C}_{n}=\{c|c\in[n]~and\mod(c,r+1)\neq0\},\\
&\mathcal{C}^{R}_{n}=\{n+1-c|c\in \mathcal{C}_{n}\},\\
&\mathcal{C}'_{n}=\{c|c\in\mathcal{C}_{n}~and\mod(c,r+1)\leq k\},\\
&\mathcal{C}^{*}_{n}=\{c|c\in[n]~and\mod(c,r+1)=0\},\\
&\mathcal{C}^{*R}_{n}=\{n+1-c|c\in \mathcal{C}^{*}_{n}\}, and\\
&\mathcal{C}^{**}_{n}=\mathcal{C}_{n}\setminus (\mathcal{C}^{*R}_{n}\cup \mathcal{C}'_{n})).
\end{split}
\end{equation}

\noindent
Then it is easy to see that the following facts hold:
\begin{equation} \label{4.2}
\begin{split}
&[n]=\mathcal{C}_{n}\cup\mathcal{C}^{*}_{n},\\
&\mathcal{C}_{n}=\mathcal{C}^{*R}_{n}\cup \mathcal{C}'_{n}\cup\mathcal{C}^{**}_{n},\\
&\mathcal{C}'_{n}\cup\mathcal{C}^{**}_{n}=\mathcal{C}_{n}\cap\mathcal{C}^{R}_{n},and\\
&\mathcal{C}^{**}_{n}=[n]\setminus (\mathcal{C}^{*R}_{n}\cup \mathcal{C}'_{n}\cup\mathcal{C}^{*}_{n}).
\end{split}
\end{equation}

Note that if $k=0$, then $\mathcal{C}'_{n}=\emptyset$,
if $k=r-1$, then $\mathcal{C}^{**}_{n}=\emptyset$,
if $k=r$, then $\mathcal{C}'_{n}=\mathcal{C}_{n}$,
$\mathcal{C}^{*}_{n}=\mathcal{C}^{*R}_{n}$ and $\mathcal{C}^{**}_{n}=\emptyset$.
Assume $u=u_{1}u_{2}\ldots u_{n}\neq0^{n}$ is any vertex of $O\Gamma_{n}^{(p,r)}$,
and the set of the coordinates of the element $1$s in $u$ is denoted as $\mathcal{X}$.
Obviously, $|\mathcal{X}|=w(u)\geq1$.
Suppose that $u_{t_{1}}=\ldots=u_{t_{|\mathcal{X}|}}=1$.
Now we show that $e(u)\geq e(0^{n})$,
and ``='' holds if and only if the following conditions hold at the same time:

$(a)$ $t_{j}\in\mathcal{C}'_{n}$ for all $j=1,\ldots,|\mathcal{X}|$;

$(b)$ $k\neq0$, and $1\leq |\mathcal{X}|\leq k$; and

$(c)$ for $|\mathcal{X}|\geq2$ and $t_{i}<t_{i+1}$, $\mod(t_{i},r+1)<\mod(t_{i+1},r+1)$.

\noindent
It is clear that if $(a),(b)$ and $(c)$ are proven,
then we known that $\mathcal{Z}(O\Gamma_{n}^{(1,r)})=\{0^{n}\}$ for the case $k=0$.
Further, by $(a),(b)$ and $(c)$ we can determine $\mathcal{Z}(O\Gamma_{n}^{(1,r)})$ and $|\mathcal{Z}(O\Gamma_{n}^{(1,r)})|$ for the case $k\neq0$ as follows.

Assume that $\mathcal{Y}=\mathcal{X}\cap\mathcal{C}^{*}_{n}\neq\emptyset$.
Then $|\mathcal{Y}|>0$, and
we consider the word $u^{*}$ obtained from $u$ by replacing the $|\mathcal{X}|$ elements $1$ with $0$ and the elements in the coordinates belong to $\mathcal{C}_{n}\setminus \mathcal{X}$ with $1$.
It is clear that $u^{*}\in V(O\Gamma_{n}^{(1,r)}£©)$ and $d(u,u^{*})=e(0^{n})+|\mathcal{Y}|$.
Now we assume that $\mathcal{Y}'=\mathcal{X}\cap\mathcal{C}^{*R}_{n}\neq\emptyset$,
and so $|\mathcal{Y}'|>0$.
Then we consider the word $u^{*}$$'$ obtained from $u$ by replacing the $|\mathcal{X}|$ elements $1$ with $0$ and the elements in the coordinates belong to $\mathcal{C}^{R}_{n}\setminus \mathcal{X}$ with $1$.
It is easy to see that $u^{*}$$'$$\in V(O\Gamma_{n}^{(1,r)}£©)$ and $d(u,u^{*}$$'$)$=e(0^{n})+|\mathcal{Y}'|$.
So if $w(u)>0$ and $\mathcal{X}\cap(\mathcal{C}^{*}_{n}\cup\mathcal{C}^{*R}_{n})\neq\emptyset$,
then $e(u)>e(0^{n})$.

Now we assume that $\mathcal{U}=\mathcal{X}\cap\mathcal{C}^{**}_{n}\neq\emptyset$.
Then $|\mathcal{U}|>0$.
We chose any $t_{i}\in\mathcal{X}\cap\mathcal{C}^{**}_{n}$,
and so $u_{t_{i}}=1$.
We can consider the vertex $u^{**}$ obtained from $u$ by replacing the $|\mathcal{X}|$ elements $1$ with $0$ and the elements of the coordinates belonging to $(\mathcal{C}^{R}_{n}\setminus \{1,2,\ldots,t_{i}\})\cup (\mathcal{C}_{n}\setminus \{t+1,t+2,\ldots,n\}$) with $1$.
Obviously,
$u^{**}\in V(O\Gamma_{n}^{(p,r)})$,
$u$ and $u^{**}$ agree in $[\frac{n}{r+1}]-1$ coordinates.
Note that $0^{n}$ and $v$ agree in $[\frac{n}{r+1}]$ coordinates,
and so $d(u,u^{**})=e(0^{n})+1=\lceil\frac{nr}{r+1}\rceil+1$.
Hence,
$e(u)>e(0)$ in this case.

From the above discussion,
we know that the following claim holds.

{\bf Claim 1.} \emph{If $\mathcal{X}\cap(\mathcal{C}^{*}_{n}\cup\mathcal{C}^{*R}_{n}\cup\mathcal{C}^{**}_{n})\neq\emptyset$,
then $e(u)>e(0^{n})$}.

By Eq. \ref{4.2},
$[n]=\mathcal{C}^{*R}_{n}\cup \mathcal{C}'_{n}\cup\mathcal{C}^{**}_{n}\cup\mathcal{C}^{*}_{n}$.
For the case $k=0$, we know $\mathcal{C}'_{n}=\emptyset$ and so
$[n]=(\mathcal{C}^{*R}_{n}\cup\mathcal{C}^{*}_{n}\mathcal\cup{C}^{**}_{n})$ by the above Claim.
Hence, $e(u)>e(0^{n})$ for all $u\neq0^{n}$,
and so $0^{n}$ is the only central vertex of $O\Gamma_{n}^{(p,r)}$ if $k=0$.

Now we turn to consider the case $k>0$.
By Eqs. (\ref{4.1}) and (\ref{4.2}),
we know that $\mathcal{C}'_{n}\neq\emptyset$.
If there exists central vertex $u$ such that $w(u)>0$,
then $\mathcal{X}\subseteq \mathcal{C}'_{n}$ must holds by Claim 1.
Hence, $(a)$ holds.
We further can find all the central vertices by showing $(b)$ and $(c)$ hold as follows.

First for the case $|\mathcal{X}|=1$,
we consider the vertex $u'$ obtained from $u$ by replacing this element $1$ with $0$ and the elements in the coordinates belong to $\mathcal{C}_{n}\setminus \{t_{1}\}$ with $1$.
Since there at least $[\frac{n}{r+1}]$ zeros in any vertex of length $n$,
it is obvious that $e(u)=d(u,u')=e(0^{n})=\lceil\frac{nr}{r+1}\rceil$.
Now we consider the case $|\mathcal{X}|\geq 2$.
So $u_{t_{1}}=u_{t_{2}}=\ldots=u_{t_{|\mathcal{X}|}}=1$,
where $t_{1}< t_{2}<\ldots< t_{|\mathcal{X}|}$ and $|\mathcal{X}|\geq2$.
We have the following claim:

{\bf Claim 2.} \emph{If $|\mathcal{X}|\geq 2$,
then $mod(t_{i},r+1)<mod(t_{i+1},r+1)$ for any $i\in\{1,\ldots,|\mathcal{X}|-1\}$}.

In fact,
if there exist some $i$ such that $mod(t_{i},r+1)>mod(t_{i+1},r+1)$,
then we consider the word $u''$ obtained from $u$ by the following transformation:

\textcircled{\small{1}} changing the element $1$ in $u$ to $0$;

\textcircled{\small{2}} the elements in the coordinates belong to $\mathcal{C}_{n}\setminus \{t_{1},\ldots,t_{2}\}$ with $1$; and

\textcircled{\small{3}}  the elements in the coordinates belong to $\{t_{1},\ldots,t_{2}\}\setminus\{t_{1},t_{2}\}$ with the element in the coordinates of
$\{1,\ldots,t_{2}-t_{1}-1\}$ of $v$.

As we assume that $mod(t_{i},r+1)>mod(t_{i+1},r+1)$,
the number of the disagree coordinate of $u$ and $u''$ is one more than that between $0^{n}$ and $v$.
Hence,
$d(u,u'')=e(0^{n})+1$,
and so $e(u)>e(0^{n})$.
As $t_{i}\in \mathcal{X}$ and $mod(t_{i},r+1)\leq k$,
there at most $k$ elements $1$ in $u$.
This means that $1\leq x=|\mathcal{X}|\leq k$.
So $(b)$ and $(c)$ hold.

Finally we show that the vertex $u$ such that satisfied with $(a),(b)$ and $(c)$ is a central vertex.
We consider the word $v'$ obtained from $u$ by replacing the element $1$ in $u$ with $0$,
and the elements in the coordinates belong to $\mathcal{C}_{n}\setminus \mathcal{X}$ with $1$.
Since the coordinates of the element $1$s belong to $\mathcal{X}\subseteq \mathcal{C}'_{n}$ and $\mathcal{C}'_{n}\bigcap\mathcal{C}^{*}_{n}=\emptyset$,
the number of the element $0$s in $v'$ is $|\mathcal{X}|+[\frac{n}{r+1}]$.
As all the element $1$s in $u$ have been changed to $0$,
$v'$ is the vertex with the biggest distance from $u$,
and so $d(u,v')=e(u)=e(0^{n})$.
Hence, $u$ is a central vertex such that $w(u)>0$.

Finally, based on $(a),(b)$ and $(c)$
we construct $\mathcal{Z}(O\Gamma_{n}^{(1,r)})$ and determine $|\mathcal{Z}(O\Gamma_{n}^{(1,r)})|$ for $k>0$.
Suppose $u\in\mathcal{Z}(O\Gamma_{n}^{(1,r)})$ and $u_{t_{1}}=\ldots=u_{t_{|\mathcal{X}|}}=1$.
Then all the element $0$s in $u$ are divided into $|\mathcal{X}|+1$ factors by the $|\mathcal{X}|$ element $1$s.
Assume that the length of the $i$th factor consisting with $0$s denoted is $m_{i}$,
$i=1,2,\ldots,|\mathcal{X}|+1$.
Note that here,
if there exists no element $0$ between the $i$th and $(i+1)$th of $1$s
(or before the first element of $1$, or behind the last element of $1$),
then the $(i+1)$th factor is empty and so $m_{i+1}=0$.
Suppose $\mod(t_{j},r+1)=k_{j}$ for $j=1,\ldots,|\mathcal{X}|$.
Then it is easy to find the following equation holds for $k\geq 1$:
\begin{equation}\label{4.3}
\begin{split}
&\sum\limits_{i=1}^{x+1}\mod(m_{i},r+1)\\
&=\mod(n-t_{|\mathcal{X}|},r+1)+\sum\limits_{i=1}^{|\mathcal{X}|-1}\mod(t_{i+1}-t_{i}-1,r+1)+\mod(t_{1}-1,r+1)\\
&=(k-k_{|\mathcal{X}|})+\sum\limits_{i=1}^{|\mathcal{X}|-1}(k_{i+1}-k_{i}-1)+(k_{1}-1)\\
&=k-|\mathcal{X}|\\
\end{split}
\end{equation}

\noindent
Eq. (\ref{4.3}) means that a central vertex $u$ of $O\Gamma_{n+r+1}^{(1,r)}$ can be obtained from some central vertex of $O\Gamma_{n}^{(p,r)}$ by adding the factor $0^{r+1}$ before $u$ or behind $u$,
or inserting $0^{r+1}$ into two elements $1$ of $u$,
that is,
$\mathcal{Z}(O\Gamma_{n+r+1}^{(1,r)})=\mathcal{Z}(O\Gamma_{n}^{(1,r)})^{\vee}$.
Let $\mathcal{Z}^{t}(O\Gamma_{n}^{(1,r)})$ be the subset of $\mathcal{Z}(O\Gamma_{n}^{(1,r)})$ containing exactly $t$ elements $1$,
where $t=0,\ldots,k$.
Further let $\mathcal{Z}^{t}_{s}(O\Gamma_{n}^{(1,r)})$ be the subset of $\mathcal{Z}^{t}(O\Gamma_{n}^{(1,r)})$ that exactly containing $s$ elements $1$ such that the coordinates belonged to $\{n-k+1,\ldots,n\}$,
where $s=0,\ldots,t.$
Then
\begin{equation}\label{4.4}
\begin{split}
&\mathcal{Z}(O\Gamma_{n}^{(1,r)})=\cup_{t=0}^{k}\mathcal{Z}^{t}(O\Gamma_{n}^{(1,r)}),\\
&\mathcal{Z}^{t}(O\Gamma_{n}^{(1,r)})=\cup_{s=0}^{t}\mathcal{Z}^{t}_{s}(O\Gamma_{n}^{(1,r)}).
\end{split}
\end{equation}

\noindent
Let $\mathcal{Z}^{t}_{s}(O\Gamma_{n}^{(p,r)}))^{\vee}_{t-j}$ denote the set obtained from $\mathcal{Z}^{t}_{s}(O\Gamma_{n}^{(1,r)})$ by inserting the factor $0^{r+1}$ immediately behind the $(t-j)$th $1$ in the vertex of $\mathcal{Z}^{t}_{s}(O\Gamma_{n}^{(1,r)})$,
where $0\leq j\leq s$.
Then $\mathcal{Z}^{t}_{s}(O\Gamma_{n}^{(1,r)}))^{\vee}_{t-j} \subseteq \mathcal{Z}^{t}_{j}(O\Gamma_{n+r+1}^{(1,r)})$, and so for any $0\leq s\leq t\leq k$, we have

\begin{align}\label{4.5}
\mathcal{Z}^{t}_{s}(O\Gamma_{n+r+1}^{(1,r)})=(\cup_{j=s}^{t}\mathcal{Z}^{t}_{j}(O\Gamma_{n}^{(1,r)}))^{\vee}_{t-j}.
\end{align}

\noindent
By Eqs. (\ref{4.4}) and (\ref{4.5}),
we have the following iterative formula:

\begin{align}\label{4.6}
\mathcal{Z}(O\Gamma_{n+r+1}^{(1,r)})
=\cup_{t=0}^{k}(\cup_{s=0}^{t}((\cup_{j=s}^{t}\mathcal{Z}^{t}_{j}(O\Gamma_{n}^{(1,r)}))^{\vee}_{t-j}))). \end{align}

Recall that $mod(n,r+1)=k\leq r$,
so $O\Gamma_{k}^{(1,r)}\cong Q_{k}$ by Proposition \ref{prop2.3}.
It is well known that $\mathcal{Z}(Q_{k})=V(Q_{k})$.
Hence, the center of $\mathcal{Z}(O\Gamma_{k}^{(1,r)})$ is obviously:

\begin{align}\label{4.7}
&\mathcal{Z}(O\Gamma_{k}^{(1,r)})=\cup_{t=0}^{k}\mathcal{Z}^{t}(O\Gamma_{k}^{(1,r)}),
\end{align}

\noindent
where
$\mathcal{Z}^{t}(O\Gamma_{k}^{(1,r)})=\{u|u\in V(Q_{k}),w(u)=t\}$, and $|\mathcal{Z}^{t}(O\Gamma_{k}^{(1,r)})|= \tbinom{k}{t}$.
Note that if $n=k$,
then $\{n-k+1,\ldots,n\}=\{1,\ldots,k\}$.
So the facts $\mathcal{Z}^{t}_{s}(O\Gamma_{k}^{(1,r)})=\emptyset$ for $s<t$
and $\mathcal{Z}^{t}(O\Gamma_{k}^{(1,r)})=\mathcal{Z}^{t}_{t}(O\Gamma_{k}^{(1,r)})$ hold.
Hence, Eq. (\ref{4.7}) is the initial set of the iterative formula (\ref{4.6}).
These facts also show that $|\mathcal{Z}^{t}_{s}(O\Gamma_{n}^{(1,r)})|$ is a multiple of $|\mathcal{Z}^{t}(O\Gamma_{k}^{(1,r)})|= \tbinom{k}{t}$ by Eq. (\ref{4.5}),
and so we can suppose that $|\mathcal{Z}^{t}_{s}(O\Gamma_{n}^{(1,r)})|=\tbinom{k}{t}p^{t}_{s}(O\Gamma_{n}^{(1,r)})$,
where $p^{t}_{s}(O\Gamma_{n}^{(1,r)})$ is some nonnegative integer.
Then Eqs.(\ref{4.4})-(\ref{4.6}) show that

\begin{equation}\label{4.8}
\begin{split}
&|\mathcal{Z}^{t}_{s}(O\Gamma_{n+r+1}^{(1,r)})|=\sum_{j=s}^{t}|\mathcal{Z}^{t}_{j}(O\Gamma_{n}^{(1,r)})|
=\tbinom{k}{t}(\sum_{j=s}^{t}p^{t}_{j}(O\Gamma_{n}^{(1,r)})),\\
&|\mathcal{Z}^{t}(O\Gamma_{n+r+1}^{(1,r)})|=\tbinom{k}{t}\sum_{s=0}^{t}(\sum_{j=s}^{t}p^{t}_{j}(O\Gamma_{n}^{(1,r)})),\\ &|\mathcal{Z}(O\Gamma_{n+r+1}^{(1,r)})|
=\sum_{t=0}^{k}(\tbinom{k}{t}\sum_{s=0}^{t}(\sum_{j=s}^{t}p^{t}_{j}(O\Gamma_{n}^{(1,r)}))).
\end{split}
\end{equation}

\noindent
Obviously,
the initial values of Eqs. (\ref{4.8}) are:

\begin{align}\label{4.9}
p^{t}_{j}(O\Gamma_{n}^{(1,r)})=0~for~j<t,
~and~p^{t}_{t}(O\Gamma_{n}^{(1,r)})=1.
\end{align}

\noindent
By the first equation of Eqs. (\ref{4.8}),
$|\mathcal{Z}^{t}(O\Gamma_{n}^{(1,r)})|$ is also a multiple of $\tbinom{k}{t}$
and
\begin{align}\label{4.10}
p^{t}_{j}(O\Gamma_{n+r+1}^{(1,r)})=\sum_{j=s}^{t}p^{t}_{j}(O\Gamma_{n}^{(1,r)}).
\end{align}

\noindent
By Eq. (\ref{4.10}) and the second equation of (\ref{4.8}),
the series of numbers $|\mathcal{Z}^{t}(O\Gamma_{k}^{(1,r)})|$,
$|\mathcal{Z}^{t}(O\Gamma_{k+(r+1)}^{(1,r)})|$,
$|\mathcal{Z}^{t}(O\Gamma_{k+2(r+1)}^{(1,r)})|$,
\ldots,
is an arithmetic progression of order $t$.
So further by the third equation of (\ref{4.8}) we know that
the series of numbers $|\mathcal{Z}(O\Gamma_{k}^{(1,r)})|$,
$|\mathcal{Z}(O\Gamma_{k+(r+1)}^{(1,r)})|$,
$|\mathcal{Z}(O\Gamma_{k+2(r+1)}^{(1,r)})|$,
\ldots,
is an arithmetic progression of order $k$.
Hence,
we know that there exist $k+1$ numbers $a_{1},a_{2},\ldots,a_{k},a_{k+1}$ such that
$|\mathcal{Z}(O\Gamma_{n}^{(1,r)})|=a_{1}\lceil\frac{n}{r+1}\rceil^{k}+a_{2}\lceil\frac{n}{r+1}\rceil^{k-1}+\ldots+a_{k}\lceil\frac{n}{r+1}\rceil +a_{k+1}$.
The value of $\mathcal{Z}(O\Gamma_{n}^{(1,r)})$ for $n=k,k+(r+1),\ldots,k+k(r+1)$ can be got by Eqs. (\ref{4.8}),(\ref{4.9}) and (\ref{4.10}).
Let
\begin{center}
\begin{spacing}{2.0}
\end{spacing}
$x=\left[
 \begin{matrix}
   a_{1}\\
   a_{2}\\
   \vdots\\
   a_{k+1}\\
  \end{matrix}
  \right]$,
$A=\left[
 \begin{matrix}
   1 & 1 & \cdots & 1\\
   2^{k} & 2^{k-1} & \cdots & 1\\
   \vdots & \vdots & \ddots & \vdots \\
   (k+1)^{k} & (k+1)^{k-1} & \cdots &1 \\
  \end{matrix}
  \right]$, and
  $b$=$\left[
 \begin{matrix}
   |\mathcal{Z}(O\Gamma_{k}^{(p,r)})|\\
   |\mathcal{Z}(O\Gamma_{k+r+1}^{(p,r)})|\\
   \vdots\\
   |\mathcal{Z}(O\Gamma_{k+k(r+1)}^{(p,r)})|\\
  \end{matrix}
  \right].$
\begin{spacing}{2.0}
\end{spacing}
\end{center}

\noindent
Then $Ax=b$.
Since $|A|$ is Vandermonde determinant,
it is nonsingular.
Hence $x=A^{-1}b$.
\end{proof}

\begin{thm}\label{thm4.2}
Let $p\geq2,r\geq1$ and $n\geq1$.
Then $rad(O\Gamma_{n}^{(p,r)})=\lceil \frac{nr}{pr+1}\rceil$, and
$\mathcal{Z}(O\Gamma_{n}^{(p,r)})$ and $|\mathcal{Z}(O\Gamma_{n}^{(p,r)})|$ can be determined as follows:

$(1)$ if $mod(n,pr+1)\neq kp+1$ for any $k\geq0$,
then $\mathcal{Z}(O\Gamma_{n}^{(p,r)})=\{0^{n}\}$,

$(2)$ if $mod(n,pr+1)=kp+1$ for some $k\geq0$,
then $\mathcal{Z}(O\Gamma_{n}^{(p,r)})=\{0^{pr+1}c,c0^{pr+1}|c\in\mathcal{Z}(O\Gamma_{n-pr-1}^{(p,r)})\}$ and $|\mathcal{Z}(O\Gamma_{n}^{(p,r)})|=([\frac{n}{pr+1}]+1)(k+1)+1$.
\end{thm}

\begin{proof}
Let $p\geq2,r\geq1$ and $n=s(pr+1)+m$ for some $s\geq0$ and $0\leq m\leq pr$.
We first show that $0^{n}$ must be a central vertex of $O\Gamma_{n}^{(p,r)}$,
and then show that $e(u)\geq e(0^{n})$ for all $u\neq 0^{n}$.
Further,
find all the vertex $u$ such that $e(u)=e(0^{n})$.

The eccentricity of $0^{n}$ can be determined as follows.
It is not difficult to see that the word $v=v_{1}v_{2}\ldots v_{n}$ such that $d(0^{n},v)=e(0^{n})$ should has the most element $1$s,
that is, $e(0^{n})=\max\limits_{v\in V(O\Gamma_{n}^{(p,r)})} w(v)$.
By Definition \ref{defn1},
there at most $r$ $1$s in every factor of length of $pr+1$ in a vertex of $O\Gamma_{n}^{(p,r)}$.
So the vertex $v=(10^{p-1})^{r}0(10^{p-1})^{r}0\ldots$ satisfies $d(0^{n},v)=e(0^{n})$.
It is easy to see that there is $r\lfloor\frac{n}{pr+1}\rfloor+\lceil\frac{m}{p}\rceil=\lceil \frac{nr}{pr+1}\rceil$ element $1$s in $v$.
Hence, $e(0^{n})=w(v)=\lceil \frac{nr}{pr+1}\rceil$.

The following sets of numbers is needed for later discussion.
Let

$[n]=\{1,2,\ldots,n\};$

$\mathcal{C}_{n}=\{c|c\in [n]~and\mod(\mod(c,pr+1),p)=1\};$

$\mathcal{C}^{R}_{n}=\{c^{R}|c^{R}=n+1-c~and~c\in \mathcal{C}_{n}\};$

$\mathcal{C}^{*}_{n}=\{c^{*}|c^{*}=c+1~and~c\in \mathcal{C}_{n}\};$ and

$\mathcal{C}'_{n}=\{c|c\in \mathcal{C}_{n}~and\mod(c,pr+1)\leq\mod(n,pr+1)\}.$

\noindent
As $p\geq2$, the following facts hold obviously:

\begin{equation}\label{4.11}
\begin{split}
&v_{c}=1, c\in\mathcal{C}_{n},\\
&v_{c^{R}}=1,c^{R}\in\mathcal{C}^{R}_{n},\\
&|\mathcal{C}_{n}|=|\mathcal{C}^{R}_{n}|=\lceil \frac{nr}{pr+1}\rceil,\\
&|\mathcal{C}^{*}_{n}|=|\mathcal{C}_{n}|~if~mod(n,pr+1)\neq kp+1,\\
&|\mathcal{C}^{*}_{n}|=|\mathcal{C}_{n}|-1 and \mathcal{C}'_{n}=\mathcal{C}^{R}_{n}\cap\mathcal{C}_{n}
~if~mod(n,pr+1)= kp+1
\end{split}
\end{equation}

Now we we turn to show that $e(u)\geq e(0^{n})$ for all $u=u_{1}u_{2}\ldots u_{n}\neq 0^{n}$,
and find all the vertex $u$ such that $e(u)=e(0^{n})$.
Suppose $\mathcal{X}=\{t|u_{t}=1\}$.
Then $|\mathcal{X}|=w(u)$ and $1\leq |\mathcal{X}|\leq \lceil \frac{nr}{pr+1}\rceil$.
We claim that $e(u)\geq e(0^{n})$,
and ``='' holds if and only if it is satisfied with the following conditions at the same time:

$(a)$ $mod(n,pr+1)=m=kp+1$ for some $k\geq0$ (of course $k\leq r-1$); and

$(b)$ $|\mathcal{X}|=1$,
and $u_{t}=1$ for some
$t\in \mathcal{C'}_{n}$.

This claim can be obtained from the following discussion.
We can first show that $e(u)>e(0^{n})$ if $\mathcal{Y}=\mathcal{X}\cap([n]\setminus\mathcal{C}_{n})\neq\emptyset$, or $\mathcal{Y}^{R}=\mathcal{X}\cap([n]\setminus\mathcal{C}^{R}_{n})\neq\emptyset$.
In fact,
if $\mathcal{Y}\neq\emptyset$,
then we consider the word $u^{*_{1}}$ obtained from $u$ by replacing the $|\mathcal{X}|$ elements $1$ with $0$ and the elements in the coordinate $c^{*_{1}}\in\mathcal{C}_{n}\setminus \mathcal{X}$ with $1$.
It is clear that $u^{*_{1}}$ is a vertex of $O\Gamma_{n}^{(p,r)}$,
and $e(u)\geq d(u,u^{*_{1}})=\lceil \frac{nr}{pr+1}\rceil+|\mathcal{Y}|>\lceil \frac{nr}{pr+1}\rceil=e(0^{n})$ by the first and the third equation of Eqs. (\ref{4.11}).
Similarly, we can show that if $\mathcal{Y}^{R}\neq\emptyset$,
then $e(u)\geq\lceil \frac{nr}{pr+1}\rceil+|\mathcal{Y}^{R}|>\lceil \frac{nr}{pr+1}\rceil=e(0^{n})$.

Obviously,
the case $\mathcal{Y}=\emptyset$ and $\mathcal{Y}^{R}=\emptyset$ actually  mean that $\mathcal{X}\subseteq \mathcal{C}_{n}\cap\mathcal{C}^{R}_{n}$.
We distinguish three subcase to continue our discussion:
$\mod(n,pr+1)\neq kp+1$; $\mod(n,pr+1)= kp+1$ and $|\mathcal{X}|\geq2$;
$\mod(n,pr+1)= kp+1$ and $|\mathcal{X}|=1$.

For the first and the second subcases,
we consider the word $u^{*_{2}}$ obtained from $u$ by replacing the $|\mathcal{X}|$ element $1$s with $0$ and the elements in the coordinate $c^{*_{2}}\in \mathcal{C}^{*}_{n}$ with $1$.
It is easy to see that the word $u^{*_{2}}\in V(O\Gamma_{n}^{(p,r)})$.
If $mod(n,pr+1)\neq kp+1$,
then $d(u,u^{*_{2}})=\lceil \frac{nr}{pr+1}\rceil+|\mathcal{X}|>e(0^{n})$ by the fourth equation of Eqs. (\ref{4.11}).
If $mod(n,pr+1)= kp+1$ and $|\mathcal{X}|\geq2$,
then $d(u,u^{*_{2}})=\lceil \frac{nr}{pr+1}\rceil-1+|\mathcal{X}|>e(0^{n})$ by the fifth equation of Eqs. (\ref{4.11}).

Finally, for the third subcase we consider the vertex $u'$ obtained from $u$ by replacing the only element $1$ with $0$,
and the elements in the coordinates belong to $\mathcal{C}_{n}\setminus\mathcal{X}$ with $1$,
that is, $u'=v+e_{t}$.
Obviously,
$u'$ is a vertex of $O\Gamma_{n}^{(p,r)}$,
and $d(u,u')=\lceil \frac{nr}{pr+1}\rceil =e(0^{n})$.
It means that $e(u)\geq e(0^{n})$.
We know that there is at most $\lceil \frac{nr}{pr+1}\rceil$ element $1$s in a vertex of $O\Gamma_{n}^{(p,r)}$ and the beginning and end elements of such vertex must be $1$ by the fact $\mod(n,pr+1)= kp+1$.
Since there only one element $1$ in $u$, $e(u)\leq \lceil \frac{nr}{pr+1}\rceil+1=e(0^{n})+1$.
We claim that $e(u)=\lceil \frac{nr}{pr+1}\rceil=e(0^{n})$.
Otherwise there exist a vertex $u^{*}$ such that $d(u,u^{*})=\lceil \frac{nr}{pr+1}\rceil+1$.
Then $w(u^{*})=\lceil \frac{nr}{pr+1}\rceil$, and $u^{*}_{t}=0$ by $u_{t}=1$.
But this is a contradiction to that $u^{*}_{t}$ should be $1$ by the fact $t\in\mathcal{C}'_{n}$.
So $e(u)= e(0^{n})$ and $u$ is also a central vertex.

By the above proof we know that if $mod(n,pr+1)\neq kp+1$,
then $0^{n}$ is the only central vertex of $O\Gamma_{n}^{(p,r)}$.
If $mod(n,pr+1)= kp+1$,
then the word $0^{n}+e(t)\neq 0^{n}$ for $t\in \mathcal{C'}_{n}$ is also a central vertex.
So there are $|\mathcal{C'}_{n}|+1=([\frac{n}{pr+1}]+1)(k+1)+1$ central vertices.
It is easy to see that for $n=s(pr+1)+kp+1$,
$n-(pr+1)=(s-1)(pr+1)+kp+1$ holds.
Hence,
$\mathcal{C'}_{n-(pr+1)}=\{c+pr+1|c\in\mathcal{C'}_{n-pr-1}\}$, and so

\begin{align}\label{4.12}
\mathcal{Z}(O\Gamma_{n}^{(p,r)})=0^{pr+1}\mathcal{Z}(O\Gamma_{n-pr-1}^{(p,r)})\cup \mathcal{Z}(O\Gamma_{n-pr-1}^{(p,r})0^{pr+1}.
\end{align}

\noindent
This means that if the centers of $O\Gamma_{n}^{(p,r)}$ are found for $n=1,2,\ldots,pr+1$,
then all the centers for any dimension $n>pr+1$ can be determined by Eq. (\ref{4.12}).
By $(a)$ and $(b)$,
the initial set of Eq. (\ref{4.12}) can be given easily:

\begin{equation}\label{4.13}
\begin{split}
&\mathcal{Z}(O\Gamma_{n}^{(p,r)})=\{0^{n}\}~for~n\neq kp+1~and~n< kp+1,~and\\
&\mathcal{Z}(O\Gamma_{kp+1}^{(p,r)})=\{0^{kp+1},0^{kp}1,0^{(k-1)p}10^{p},\ldots,10^{kp}\}~for~n=kp+1~ ~0\leq k< r.
\end{split}
\end{equation}

Hence,
for any dimension $n$ all the central vertices of $O\Gamma_{n}^{(p,r)}$ can be found by Eqs. (\ref{4.13}) and (\ref{4.12}).
This completes the proof.
\end{proof}

By Proposition \ref{prop2.7},
The radius and center of $I\Gamma_{n}^{(1,r)}$ can be determined by Theorem \ref{thm4.1},
and of $I\Gamma_{n}^{(p,1)}$ can be determined by Theorem \ref{thm4.2}.

The problem on the radius and center of $I\Gamma_{n}^{(1,r)}$ for $p\geq2$ and $r\geq2$ has not been answered yet,
as shown in Problem \ref{pro2} in Section 6.
We think it is a tricky problem.

\subsection{Diameter}

Egiazarian and Astola determined the diameter of $O\Gamma_{n}^{(p,r)}$:

\begin{thm}\label{thm4,4}\cite{Egiazarian}
The diameter of $O\Gamma_{n}^{(p,r)}$ is

{\renewcommand\baselinestretch{0.90}\selectfont
\begin{equation}\nonumber
diam(O\Gamma_{n}^{(p,r)})=\begin{cases}
n,~~~if~ p=1,r\geq1;\\
\lfloor \frac{nr}{pr+1}\rfloor+\lfloor\frac{(n-1)r}{pr+1}\rfloor,~~~if~p>1,r\geq1.\\
\end{cases}
\end{equation}
\par}
\end{thm}

The diameters of some $I\Gamma_{n}^{(p,r)}$ are determined in this subsection.
By these results,
we know that there exist many $I\Gamma_{n}^{(p,r)}$ with $diam(I\Gamma_{n}^{(p,r)})< n$,
many $I\Gamma_{n}^{(p,r)}$ with $diam(I\Gamma_{n}^{(p,r)})=n$,
also many $I\Gamma_{n}^{(p,r)}$ with $diam(I\Gamma_{n}^{(p,r)})>n$.

\begin{thm}\label{thm4.5}
Let $p\geq1, r\geq1$ and $n\geq1$.
Then for the diameter of $I\Gamma_{n}^{(p,r)}$,
we have\\
(1) for $p=1$, diam$(I\Gamma_{n}^{(p,r)})=n$;\\
(2) for $p\geq2$,

~~(2.1) $diam(I\Gamma_{n}^{(p,r)})= 2r\lceil\frac{n}{p+r}\rceil+min\{mod(n,p+r),2r\}$ if $r< p$;

~~(2.2) $diam(I\Gamma_{n}^{(p,r)})=n$ if $p\leq r \leq p+1$;

~~(2.3) $diam(I\Gamma_{n}^{(p,r)})=n$ if $p+2\leq r\leq 2p+2$;

~~(2.4) $diam(I\Gamma_{n}^{(p,r)})=n$ if $r\geq2p+3$ and $n< 2p+3$

~~(2.5) $n+c\lfloor\frac{n}{r'+p}\rfloor\leq diam(I\Gamma_{n}^{(p,r)})\leq n+c\lceil\frac{n}{r'}\rceil$ for some $c\geq1$ and $2p+3\leq r'\leq r$ if $r\geq2p+3$ and $n\geq2p+3$.
\end{thm}

\begin{proof}
First we give a claim on the case $r\geq p$.

{\bf Claim 1.} \emph{If $r\geq p$, then $diam(I\Gamma_{n}^{(p,r)})\geq n$.}

Assume that $n=s(p+r)+t$ for some $s$ and $t$,
where $s\geq0$ and $0\leq t<p+r$.
We distinguish two cases to prove this claim: $0\leq t\leq r$ and $r<t<p+r$.
For the former case,
let
$\alpha=(1^{r}0^{p})^{s}1^{t}$, and
$\beta=(0^{r}1^{p})^{s}0^{t}$.
For the latter case, let
$\alpha=(1^{r}0^{p})^{s}1^{r}0^{t-r}$, and
$\beta=(0^{r}1^{p})^{s}0^{t}1^{t-r}$.
Then $H(\alpha,\beta)=n$ hold obviously.
As $r\geq p$,
both $\alpha$ and $\beta$ are vertices of $I\Gamma_{n}^{(p,r)}$.
Hence, $diam(I\Gamma_{n}^{(p,r)})\geq n$.
Claim 1 is proved.

It is well known that $diam(Q_{n})=n$ for any $n\geq1$.
So if $I\Gamma_{n}^{(p,r)}$ is an isometric subgraph of $Q_{n}$,
then $diam(I\Gamma_{n}^{(p,r)})\leq n$.
Theorem 1.1 in \cite{JZ} show that if $p=1$,
or $p\geq2$ and $p\leq r\leq p+1$,
then $I\Gamma_{n}^{(p,r)}$ is an isometric subgraph of $Q_{n}$,
and so $diam(I\Gamma_{n}^{(p,r)})\leq n$.
By Claim 1, $diam(I\Gamma_{n}^{(p,r)})\leq n$ for $r\geq p$.
Hence, $(1)$ and $(2.2)$ hold.

Next we consider $(2.1)$.
By Theorem 1.1 \cite{JZ},
$I\Gamma_{n}^{(p,r)}$ is an isometric subgraph of $Q_{n}$ if $p\geq2$ and $r< p$.
So for any two vertexes of $I\Gamma_{n}^{(p,r)}$,
the distance of them in $I\Gamma_{n}^{(p,r)}$ is the hamming distance between them.
We distinguish four cases to continue the discussion:
\textcircled{\small{1}} $n<r$,
\textcircled{\small{2}} $r\leq n\leq2r$,
\textcircled{\small{3}} $2r< n\leq p+r$ and
\textcircled{\small{4}} $n>r+p$.
For \textcircled{\small{1}},
let
$\alpha=1^{n}$, and
$\beta=0^{n}$.
For \textcircled{\small{2}}, let
$\alpha=0^{n-r}1^{r}$, and
$\beta=1^{n-r}0^{r}$;
Then $H(\alpha,\beta)=n$,
and so $diam(I\Gamma_{n}^{(p,r)})=n$ for \textcircled{\small{1}} and \textcircled{\small{2}}.
If \textcircled{\small{3}} happens,
then there at most one factor $1^{r}$ in any vertex of $I\Gamma_{n}^{(p,r)}$.
So the following words have the longest hamming distance:
$\alpha=1^{r}0^{n-r}$, and
$\beta=0^{n-r}1^{r}$.
Obviously, $H(\alpha,\beta)=2r$.
For \textcircled{\small{4}} $n>p+r$,
let $n=s(p+r)+k$ for some $s\geq1$ and $0\leq k<p+r$.
That is $s=\lfloor\frac{n}{p+r}\rfloor$ and $k=mod(n,p+r)$.
Suppose that $d(\alpha,\beta)=diam(I\Gamma_{n}^{(p,r)})$.
Then any factor of $\alpha$ and $\beta$ of length $p+r$ may differ at most in $2r$ bits,
and so $\alpha$ and $\beta$ should be the words as shown in the following:
$\alpha=\ldots1^{r}0^{r}0^{p-r}\ldots$, and
$\beta=\ldots0^{r}1^{r}0^{p-r}\ldots$.
Obviously, if $k<2r$,
then $H(\alpha,\beta)=2r\lceil\frac{n}{p+r}\rceil+k$,
and if $k\geq2r$,
then $diam(I\Gamma_{n}^{(p,r)})=2r\lceil\frac{n}{p+r}\rceil+2r$.
So for \textcircled{\small{4}},
$diam(I\Gamma_{n}^{(p,r)})=2r\lceil\frac{n}{p+r}\rceil+min\{mod(n,p+r),2r\}$.
It is easy to see that the value of $diam(I\Gamma_{n}^{(p,r)})$ in \textcircled{\small{1}}, \textcircled{\small{2}} and \textcircled{\small{3}} also can be write as the form as it in \textcircled{\small{4}}.
So $(2.1)$ holds.

Now we turn to $(2.3)$, $(2.4)$ and $(2.5)$.
If $10^{t}1$ and $11^{t}1$ appearing in the same coordinates of vertices $\alpha$ and $\beta$ of $I\Gamma_{n}^{(p,r)}$,
then we call there is a distance-barrier in $\alpha$ and $\beta$.
It has been shown that
$d_{I\Gamma_{n}^{(p,r)}}(\mu,\nu)=H(\mu,\nu)$ if and only if there does not exist any distance-barrier in $\mu$ and $\nu$ \cite{JZ}.
Clearly, a distance-barrier appears if and only if $p\geq2$, $r\geq p+2$ and $n\geq p+2$.
This concept can be easily generalized as follows.
If for the case $p\geq2$, $r\geq sp+(s+1)$ and $n\geq sp+(s+1)$,
\begin{align}\label{4.14}
\alpha'=1^{r_{1}}0^{p_{1}}1^{r_{2}}\ldots0^{p_{s}}1^{r_{s+1}},~and~
\beta'=1^{r_{1}}1^{p_{1}}1^{r_{2}}\ldots1^{p_{s}}1^{r_{s+1}},
\end{align}
\noindent
appears in the same coordinates of vertices $\alpha,\beta$ of $I\Gamma_{n}^{(p,r)}$,
then it is called there is a $s$-distance-barrier in $\alpha$ and $\beta$,
where $r_{i}\geq 1$, $p_{j}\geq p$, $\sum_{i=1}^{s+1}r_{i}+\sum_{j=1}^{s}p_{j}=r'\leq r$.
Obviously, the distance-barrier introduced earlier is a $s$-distance-barrier of the case $s=1$.

By Claim 1, we know that $diam (I\Gamma_{n}^{(p,r)})\geq n$ if $r\geq p+2$.
Since there may exist $s$-distance-barrier in some vertices of $I\Gamma_{n}^{(p,r)}$ if $r\geq p+2$,
perhaps $diam (I\Gamma_{n}^{(p,r)})> n$ holds.
Suppose that there is a $s_{t}$-distance-barrier appears in vertices $\alpha,\beta$ of $I\Gamma_{n}^{(p,r)}$,
where $t=1,\ldots,k$ for some $k\geq1$ and the $s_{t}$-distance-barrier is
$\alpha'_{t}=1^{r_{t_{1}}}0^{p_{t_{1}}}1^{r_{t_{2}}}\ldots0^{p_{s_{t}}}1^{r_{s_{t}+1}}$,
$\beta'_{t}=1^{r_{t_{1}}}1^{p_{t_{1}}}1^{r_{t_{2}}}\ldots1^{p_{s_{t}}}1^{r_{s_{t}+1}}$, and
$r_{t_{i}}\geq 1$, $p_{t_{j}}\geq p$, $\sum_{i=1}^{s_{t}+1}r_{t_{i}}+\sum_{j=1}^{s_{t}}p_{t_{j}}=r_{t}'\leq r$.
Then
\begin{align}\label{4.15}
d_{I\Gamma_{n}^{(p,r)}}(\alpha,\beta)-n=\sum^{k}_{t=1}d_{I\Gamma_{r'_{t}}^{(p,r)}}(\alpha'_{t},\beta'_{t})-r'_{t},
\end{align}
\noindent
and $d_{I\Gamma_{r'_{t}}^{(p,r)}}(\alpha'_{t},\beta'_{t})-r'_{t}$ is called the contribution to $d_{I\Gamma_{n}^{(p,r)}}(\alpha,\beta)$ of the $s_{t}$-distance-barrier.
The problem on determining $diam(I\Gamma_{n}^{(p,r)})$ becomes to find the $s_{t}$-distance-barrier with the maximum value of $d_{I\Gamma_{r'_{t}}^{(p,r)}}(\alpha'_{t},\beta'_{t})-r'_{t}$ for $t=1,\ldots,k$ by Eq. (\ref{4.15}).

Now we calculate $d_{I\Gamma_{r'}^{(p,r)}}(\alpha',\beta')-r'$ for the $s$-distance-barrier shown in (\ref{4.14}).
Assume that $r_{k}=max\{r_{1},r_{2},\ldots,r_{s+1}\}$.
As it should be avoided to appear more than $r$ consecutive 1s and less than $p$ 0s between two
factors composed of (at most $r$) consecutive 1s when $\alpha'$ is transformed into $\beta'$,
the factor $1^{r_{i}}$  are changed to $0^{r_{i}}$ and then to $1^{r_{i}}$, $i=1,\ldots,k-1,k+1,\ldots,s+1$.
Hence, $d(\alpha',\beta')=2(r_{1}+\ldots+r_{k-1}+r_{k+1}+\ldots+r_{s+1})+p_{1}+\ldots+p_{s}$.
So the contribution to $d(\alpha',\beta')$ of this $s$-distance-barrier is
\begin{align}\label{4.16}
d(\alpha',\beta')-r'=r_{1}+\ldots+r_{k-1}+r_{k+1}+\ldots+r_{s+1}-r_{k}.
\end{align}
For $p\geq2$ and $p+2\leq r\leq2p+2$, or $p\geq2$, $r\geq2p+3$ but $n<2p+3$,
we know that the only $s$-distance-barrier are $1^{r_{1}}0^{p}1^{r_{2}}$ and $1^{r_{1}}1^{p}1^{r_{2}}$,
where $r_{1}\geq1,r_{2}\geq1$ and $r_{1}+p+r_{2}\leq r$.
Without loss of generality,
suppose $r_{1}=max\{r_{1},r_{2}\}$.
Then $r_{1}\geq r_{2}$,
and so $d(\alpha',\beta')-r'=r_{2}-r_{1}\leq0$,
which means that $diam(I\Gamma_{n}^{(p,r)})\leq n$.
So we know that $diam(I\Gamma_{n}^{(p,r)})=n$ by Claim 1,
and $(2.3)$ and $(2.4)$ hold.

It is clear that if $p\geq2$, $r\geq 2p+3$ and $n\geq 2p+3$,
then there always exist $s$-distance-barrier such that the contribution (Eq. (\ref{4.16})) is bigger than zeros.
For example, $10^{p}10^{p}1$ and $11^{p}11^{p}1$ is $2$-distance-barrier,
with $r_{1}=r_{2}=r_{3}=1$. Obviously, $r_{1}+r_{2}-r_{3}=1>0$.
So Claim 2 holds by the result above.

{\bf Claim 2.} \emph{Assume $p\geq1$, $r\geq 1$ and $n\geq1$.
Then $diam(I\Gamma_{n}^{(p,r)})>n$ if and only if $r\geq 2p+3$ and $n\geq 2p+3$.}

Note that for a $s$-distance-barrier in (\ref{4.14}),
if we want to maximize $d(\alpha',\beta')-r'$,
then it should be satisfied that $t_{1}=\ldots=t_{s}=p$ and the values of $r_{1},r_{2}\ldots,r_{s+1}$ also probably close to be equal by Eq. (\ref{4.16}).

Finally,
we turn to consider the case $p\geq2$, $r\geq 2p+3$ and $n\geq2p+3$.
By Claim 2, Eqs. (\ref{4.15}) and (\ref{4.16}),
if we find the $s$-distance-barrier with the biggest contribution to the distance,
then the diameter of $I\Gamma_{n}^{(p,r)}$ is found.
Although the general structure of the $s$-distance-barrier with biggest contribution is known by the above discussions,
the value of $s$ is not determined yet.
Note that for the given $p$ and $r$,
the corresponding $s$ must exist and can be found from a finite number of times comparisons.
So we only give a bound on the diameter of $I\Gamma_{n}^{(p,r)}$ for this case.

Suppose the following $(k+2)$-distance-barrier with the biggest contribution but with the smallest length:
\begin{align}\label{4.17}
\alpha'=1^{t}0^{p}1^{t}\ldots0^{p}1^{t'}~and~
\beta'=1^{t}1^{p}1^{t}\ldots1^{p}1^{t'},
\end{align}
\noindent
where $1\leq t'\leq t$,
and there are $k+2$ factors $1^{t}0^{p}$ in $\alpha$.
Then $2p+3\leq r'=(k+2)(p+t)+t'\leq r$,
$0\leq k\leq\lfloor\frac{r}{t+p}\rfloor-2$ and $1\leq t\leq\lfloor\frac{r-2p-1}{2}\rfloor$ by $r\geq2p+3$.
By Eq. (\ref{4.16}),
that the contribution of the above $(k+2)$-distance-barrier is $c=kt+t'$.
For two $(k+2)$-distance-barriers as shown (\ref{4.17}) in vertices of $I\Gamma_{n}^{(p,r)}$,
their position relationship may be
$\alpha'\beta'$ and $\beta'\alpha'$, or
$\alpha'1^{p}\beta'$ and $\beta'0^{p}\alpha'$.
So the number of $(k+2)$-distance-barrier in vertices of $I\Gamma_{n}^{(p,r)}$ is at least $\lfloor\frac{n}{r'+p}\rfloor$,
and at most is $\lceil\frac{n}{r'}\rceil$.
By Eq. (\ref{4.15}),
we know that
$n+c\lfloor\frac{n}{r'+p}\rfloor\leq diam(I\Gamma_{n}^{(p,r)})\leq n+c\lceil\frac{n}{r'}\rceil,$
and so $(2.5)$ holds.
This completes the proof.
\end{proof}

Note that for $(2.5)$ in Theorem \ref{thm4.5},
the diameter of $I\Gamma_{n}^{(p,r)}$ is bigger than $n$,
but the bounds may be sharp.
For example suppose $p=2$,
we consider the following cases,

(1) $r=7$ and $n=9x-2$ for $x\geq1$;

(2) $r=9$ and $n=7y$ for $y\geq1$.

It can be find that all the distance-barrier with the length not more than $9$ are:
$1001001$ and $1111111$; 
$10010001$ and $11111111$;
$10001001$ and $11111111$;
$100010001$ and $111111111$;
$110011001$ and $111111111$;
$110010011$ and $111111111$; and
$100110011$ and $111111111$.

Obviously, the contribution of any one of the above distance-barrier is $1$,
and the smallest length is $7$.
So for both examples (1) and (2),
$c=1$ and $r'=7$.

For (1), let
$\alpha=(100100111)^{(x-1)}(1001001)$, and
$\beta=(111111100)^{(x-1)}(1111111)$.
So this example,
the lower bound of $(2.5)$ holds.

For (2), if $y$ is even, then let
$\alpha=(10010011111111)^{\frac{y}{2}}$, and
$\beta=(11111111001001)^{\frac{y}{2}}$;
if $y$ is odd, then let
$\alpha=(10010011111111)^{\frac{y-1}{2}}(1001001)$, and
$\beta=(11111111001001)^{\frac{y-1}{2}}(1111111)$.
So for this example,
the upper bound of $(2.5)$ holds.

\smallskip
\section{Maximum and  minimum degree}

Recall that the number of neighbors of a vertex $v$ in a graph $G$ is called the degree of $v$ and denoted $d_{G}(v)$ (or simply $d(v)$).
Let $\Delta(G)$ and $\delta(G)$ be the maximum and  minimum degree of a graph $G$, respectively.
The neighbourhood $N(u)$ of a vertex $u$ is the set of vertices adjacent to $u$.
Since for $n=1$ both the graphs $O\Gamma_{n}^{(p,r)}$ and $I\Gamma_{n}^{(p,r)}$ are $K_{2}$,
we only consider the case $n\geq2$.

\begin{thm}\label{thm5.1}
Let $p,r\geq1$ and $n\geq2$.
Then $\Delta(O\Gamma_{n}^{(p,r)})=\Delta(I\Gamma_{n}^{(p,r)})=n$, and
$d_{O\Gamma_{n}^{(p,r)}}(0^{n})=d_{I\Gamma_{n}^{(p,r)}}(0^{n})=n$.
Specifically, $0^{n}$ is the only vertex with degree of $\Delta(O\Gamma_{n}^{(p,r)})$ if and only if it is one of the following cases:

$(a)$ $p\geq1,r=1$ and $n\geq2$, and
$(b)$  $p\geq2,r\geq2$ and $n\geq2$;

\noindent
$0^{n}$ is the only vertex with degree of $\Delta(I\Gamma_{n}^{(p,r)})$ if and only if it is one of the following cases:

$(a')$  $p\geq1,r=1$ and $n\geq2$,
$(b')$  $p\geq2,r=2$ and $n\geq 4, and$
$(c')$  $p\geq2,r\geq3$ and $n\geq 5$.
\end{thm}

\begin{proof}
It is well known that for every $n$,
$\Delta(Q_{n})=n$.
Since both $O\Gamma_{n}^{(p,r)}$ and $I\Gamma_{n}^{(p,r)}$ are subgraph of $Q_{n}$,
we know that $\Delta(I\Gamma_{n}^{(p,r)})\leq n$ and $\Delta(I\Gamma_{n}^{(p,r)})\leq n$ by
Proposition \ref{prop2.5}.
Let $p,r\geq1$ and $n\geq2$.
Then both $0^{n}$ and $0^{t}10^{n-t-1}$ $(t=0,1,\ldots,n-1)$ are vertices of $O\Gamma_{n}^{(p,r)}$ (and also $I\Gamma_{n}^{(p,r)}$) obviously.
It is easy to see that $0^{n}$ joint to $0^{t}10^{n-t-1}$ in $O\Gamma_{n}^{(p,r)}$ (and $I\Gamma_{n}^{(p,r)}$) for $t=0,1,\ldots,n-1$,
and so $d_{I\Gamma_{n}^{(p,r)}}(0^{n})=d_{O\Gamma_{n}^{(p,r)}}(0^{n})=n$.

First,
let us find the cubes $O\Gamma_{n}^{(p,r)}$ with only one vertex with degree $n$.
For the case $p=1,r\geq2$ and $n\geq2$,
it is easy to see that $0^{t}10^{n-t-1}+e_{s}$ is also a vertex of $O\Gamma_{n}^{(p,r)}$ for any $t\in\{0,1,\ldots,n-1\}$ and $s\in[n]$.
Therefore, $d_{O\Gamma_{n}^{(p,r)}}(0^{t}10^{n-t-1})=n$.
For the cases $p\geq1$ and $r=1$, or $p\geq2$ and $r\geq2$,
we claim that $0^{n}$ is the only vertex with the degree $\Delta(O\Gamma_{n}^{(p,r)})$.
In fact,
we may suppose $\alpha$ is any vertex of $O\Gamma_{n}^{(p,r)}$ such that $w(\alpha)>0$,
and the $t$th coordinate of $\alpha$ is $1$ for some $t\in[n]$.
Then at least one of the $\alpha+e_{t+1}$ and $\alpha+e_{t-1}$ is not a vertex of $O\Gamma_{n}^{(p,r)}$.
Therefore,
$d_{O\Gamma_{n}^{(p,r)}}(\alpha)<n$.
So $0^{n}$ is the only vertex with degree of $\Delta(O\Gamma_{n}^{(p,r)})$ if and only if one of $(a)$ and $(b)$ holds.

Next we turn to consider the cubes $I\Gamma_{n}^{(p,r)}$.
Since $I\Gamma_{n}^{(p,r)}\cong O\Gamma_{n}^{(p,r)}$ for $p=1$ or $r=1$ by Proposition \ref{prop2.7},
we know that there exist only vertex $0^{n}$ with degree $n$ for the case $p\geq1,r=1$ and $n\geq2$,
and there exist vertex $0^{t}10^{n-t-1}\neq0^{n}$ with degree $n$ for the case $p=1,r\geq2$ and $n\geq2$ by the above discussions.
The following all discussions are considered under the condition $p\geq2$.
If $r\geq2$ and $2\leq n\leq3$,
then $01$ and $010$ are the vertices with the degree of $\Delta(I\Gamma_{2}^{(p,r)})=2$ and $\Delta(I\Gamma_{3}^{(p,r)})=3$, respectively.
The rest case $r\geq2$ and $n\geq4$ can be divided to:
$r=2$ and $n\geq 4$;
$r\geq3$ and $n=4$; and
$r\geq3$ and $n\geq5$.
It is easy to see that if $r\geq3$ and $n=4$,
then $0110$ is a vertex with the degree of $\Delta(I\Gamma_{4}^{(p,r)})=4$.
We claim for the other two cases,
there exist only $0^{n}$ with the degree $\Delta(I\Gamma_{4}^{(p,r)})$.
In fact,
we may suppose $\alpha=a_{1}a_{2}\ldots a_{n}$ is any vertex of $I\Gamma_{n}^{(p,r)}$ with $w(\alpha)>0$.
First for $r=2$ and $n\geq 4$ assume that $a_{t}=1$ for some $1\leq t\leq n$.
Then at least one of $t-2$ and $t+2$ belong to $[n]$,
and so $\alpha+e_{t-2}$ or $\alpha+e_{t+2}$ is not a vertex of $I\Gamma_{n}^{(p,r)}$.
Therefore, $d_{I\Gamma_{n}^{(p,r)}}(\alpha)\leq n-1$.
Next we consider the last case $r\geq3$ and $n\geq5$,
and distinguish two subcases by the longest length $l$ of the factors consisting of element $1$s contained in $\alpha$.
If $1\leq l\leq 2$,
then at least one of $\alpha+e_{t-2}$ and $\alpha+e_{t+2}$ is not a vertex of $I\Gamma_{n}^{(p,r)}$,
and so $d_{I\Gamma_{n}^{(p,r)}}(\alpha)\leq n-1$.
Finally, we consider the case $3\leq l\leq r$.
Without loss of generality assume that $a_{t-1}=1$ and $a_{t+1}=1$.
Then $\alpha+e_{t}$ is not a vertex of $I\Gamma_{n}^{(p,r)}$.
So $d_{I\Gamma_{n}^{(p,r)}}(\alpha)\leq n-1$.
Hence,
$0^{n}$ is the only vertex with degree of $\Delta(I\Gamma_{n}^{(p,r)})$ if and only if one of $(a'),(b')$ and $(c')$ holds.
\end{proof}

For the minimum degree of $I\Gamma_{n}^{(p,r)}$ and $O\Gamma_{n}^{(p,r)}$,
we first give a lemma.

\begin{lem}\label{lem5.2}
There is a vertex $\alpha$ of $I\Gamma_{n}^{(p,r)}$ $(O\Gamma_{n}^{(p,r)})$ that $d_{I\Gamma_{n}^{(p,r)}}(\alpha)=\delta(I\Gamma_{n}^{(p,r)})$ $(=\delta(O\Gamma_{n}^{(p,r)}))$,
and all the vertices of $N(\alpha)$ can be obtained from $\alpha$ by changing some bits from $1$ to $0$.
\end{lem}

\begin{proof}
We only give the proof of $I\Gamma_{n}^{(p,r)}$ here,
and the case for $O\Gamma_{n}^{(p,r)}$ can be shown similarly.
This lemma is proven by showing that
if for a vertex $\alpha'=a'_{1}a'_{2}\ldots a'_{n}$ of $I\Gamma_{n}^{(p,r)}$ with $d_{I\Gamma_{n}^{(p,r)}}(\alpha')=\delta(I\Gamma_{n}^{(p,r)})$,
there is a vertex in $N(\alpha')$ such that obtained from $\alpha'$ by changing some $a'_{t}$ from $0$ to $1$,
then another vertex $\alpha$ such that $d_{I\Gamma_{n}^{(p,r)}}(\alpha)=\delta(I\Gamma_{n}^{(p,r)})$ can be obtained.
We distinguish three cases $r=1$, $r=2$ and $r\geq3$.
In all these cases,
we assume that the length of the block containing the element $a'_{t}=0$ such that $\alpha'+e(t)\in N(\alpha')$ is $m$.
Obviously,
this block $0^{m}$ may be the first or last block,
or between some two blocks consisting of $1$s.

{\bf Case 1}  $r=1$.

Obviously,
if $a'_{t}=0$ is contained in the first or last block,
then $m\geq p+1$.
Without loss of generality,
suppose $a'_{t}=0$ is contained in the first block.
Then $\alpha'+e(s)\in N(\alpha')$ for all $s\leq m-p$.
Further we can show that $m\geq p+2$ cannot happen by showing that $\beta=\alpha'+e_{m-(p+1)}$ is a vertex with smaller degree than $\alpha'$.
In fact,
if $p+2\leq m\leq 2p+2$,
then $d(\beta)=d(\alpha')-(m-p-1)$;
if $m> 2p+2$,
then $d(\beta)=d(\alpha')-(p+1)$.
Obviously,
this is a contradiction to the fact $d_{I\Gamma_{n}^{(p,r)}}(\alpha')=\delta(I\Gamma_{n}^{(p,r)})$.
Therefore, only the case $m=p+1$ may happen,
and so $t=1$.
Let $\alpha=\alpha'+e_{1}$.
Then $d(\alpha)=d(\alpha')=\delta(I\Gamma_{n}^{(p,r)})$.

Now we suppose the block $0^{m}$ containing $a'_{t}=0$ is between two blocks of $1$s,
that is, $\alpha'=\ldots10^{m}1\ldots$ for some $m$.
Obviously, $m\geq2p+1$ holds,
otherwise there is not $0$ in $0^{m}$ can be changed to $1$.
Assume that the first $0$ in $0^{m}$ begins from the $(k+1)$th coordinate of $\alpha'$.
Then $\alpha'+e_{s}\in N(\alpha')$ for all $k+(p+1)\leq s\leq k+(m-p)$.
Let $\beta=\alpha'+e_{p+2}$.
Then we can show that $d(\beta)<d(\alpha')$ when $m\geq2p+2$,
and so only $m=2p+1$ may happen.
In fact,
if $2p+2\leq m\leq 3p+2$,
then $d(\beta)=d(\alpha')-(m-2p-1)$;
if $m> 3p+2$,
then $d(\beta)=d(\alpha')-(p+1)$.
Obviously,
this is a contradiction to the fact $d_{I\Gamma_{n}^{(p,r)}}(\alpha')=\delta(I\Gamma_{n}^{(p,r)})$.
Therefore, $m=2p+1$ and so $t=k+p+2$.
Let $\alpha=\alpha'+e_{k+p+2}$.
Then $d(\alpha)=d(\alpha')$.

Next we turn to consider $r=2$ and $r\geq3$.
In both cases,
we show that:
if $a'_{t}=0$ is contained in the first or last block,
then $m=p+1,$ or $m=p+2$;
if $a'_{t}=0$ in some block between two blocks consisting of $1$s,
then $m=2p+1$ or  $m=2p+2$.
Those facts can be shown as follows.

{\bf Case 2}  $r=2$.

First we suppose $a'_{t}=0$ is contained in the first block (the case for the last block can be considered similarly).
Then $m\geq p+1$,
and $\alpha'+e_{s}\in N(\alpha')$ for all $s\leq m-p$.
Let $\beta=\alpha'+e_{m-p-1}+e_{m-p-2}$.
Then we can show that $m\geq p+3$ cannot happen.
In fact,
$d(\beta)=d(\alpha')-(m-p-2)$ if $p+3\leq m \leq 2p+3$,
and $d(\beta)=d(\alpha')-(p+1)$ if $2p+3< m$.
It is a contradiction to the fact $d_{I\Gamma_{n}^{(p,r)}}(\alpha')=\delta(I\Gamma_{n}^{(p,r)})$.
So only $m=p+1$ or $m=p+2$ may happens and $t=1$ or $2$.
Obviously,
both $\alpha=\alpha'+e_{1}$ and $\alpha=\alpha'+e_{1}+e_{2}$ have the same degree as $\alpha'$.
Now we suppose the case that $a'_{t}=0$ in contained some block $0^{m}$ between two blocks consisting of $1$s.
Then $m\geq 2p+1$.
Assume that the first coordinate of $0^{m}$ begins from the $k$th coordinate of $\alpha'$.
Then $\alpha'+e_{s}\in N(\alpha')$ for all $k+p\leq s\leq k+m-p$ clearly.
We claim that $m\geq 2p+3$ cannot happen.
In fact,
let $\beta=\alpha'+e_{k+p+2}+e_{k+p+3}$.
Then $d(\beta)=d(\alpha)-(m-2p-2)$ if $2p+3\leq m \leq3p+3$,
and $d(\beta)=d(\alpha)-(2p+1)$ if $3p+3<m$.
It is a contradiction to the fact $d_{I\Gamma_{n}^{(p,r)}}(\alpha')=\delta(I\Gamma_{n}^{(p,r)})$.
So only $m=2p+1$ or $m=2p+2$ happens.
Let $\alpha=\alpha'+e_{k+p+1}$ or $\alpha=\alpha'+e_{k+p+1}+e_{k+p+2}$.
Then $d(\alpha)=d(\alpha')$.

{\bf Case 3}  $r\geq3$.

Suppose $a'_{t}=0$ is contained in the first block $0^{m}$ (the last block can be considered similarly).
Then $m\geq p+1$,
and $\alpha'+e_{s}\in N(\alpha')$ for all $s\leq m-p$.
We claim that $m\geq p+3$ cannot happen.
Otherwise,
let $\beta=\alpha'+e_{m-p}+e_{m-p-1}+e_{m-p-2}$.
Then $\beta\in V(I\Gamma_{n}^{(p,r)})$ but $\beta+e_{m-p-1}\notin V(I\Gamma_{n}^{(p,r)})$,
and so $d(\alpha)>d(\beta)$.
It is a contradiction to the fact $d_{I\Gamma_{n}^{(p,r)}}(\alpha')=\delta(I\Gamma_{n}^{(p,r)})$.
So $m\leq p+2$.
For $m=p+1$ and $m=p+2$,
both the words $\alpha=\alpha'+e_{1}$ and $\alpha=\alpha'+e_{1}+e_{2}$ are vertices of $I\Gamma_{n}^{(p,r)}$
and have the same degree as $\alpha'$.
Finally,
suppose that $a'_{t}=0$ in some block $0^{m}$ between two blocks consisting of $1$s.
Then $m\geq 2p+1$.
Suppose that the first coordinate of $0^{m}$ begins from the $k$th coordinate of $\alpha'$.
Then $\alpha'+e_{s}\in N(\alpha')$ for all $k+p\leq s\leq k+m-p$ clearly.
We claim that $m\geq 2p+3$ cannot happen.
Otherwise,
let $\beta=\alpha'+e_{k+p+1}+e_{k+p+2}+e_{k+p+3}$.
Then $\beta+e_{k+p+2}\notin I\Gamma_{n}^{(p,r)}$,
and so $d(\beta)<d(\alpha)$.
It is a contradiction to the fact $\alpha'$ is the vertex with minimum degree.
For $m=2p+1$ and $m=2p+2$,
both the vertices $\alpha=\alpha'+e_{k+p+1}$ and $\alpha=\alpha'+e_{k+p+1}+e_{k+p+2}$ have the same degree as $\alpha'$.
This completes the proof.
\end{proof}

By Lemma \ref{lem5.2},
the following result holds obviously.
\begin{cor}\label{cor5.3}
For any $I\Gamma_{n}^{(p,r)}$,
there must exist a vertex with degree of $\delta(I\Gamma_{n}^{(p,r)})$ containing the minimum number of the $1$s which can be changed to $0$,
but the maximum number of coordinates of $0$s which cannot be changing to $1$.
\end{cor}
 
\begin{thm}\label{thm5.4}
Let $p,r,n\geq1$.
Then $\delta(I\Gamma_{n}^{(p,r)})$ can be determined as follows:

$(a)$ for $r=1, p\geq1$ and $n\geq1$,
$\delta(I\Gamma_{n}^{(p,r)})=\lceil \frac{n}{2p+1}\rceil$;

$(b)$ for $r\geq2, p=1$, $n\geq1$ and $t=\mod(n,2+r)$
\begin{equation*}
\delta(I\Gamma_{n}^{(p,r)})=
\begin{cases}
r[\frac{n}{r+2}]+t,~if~0\leq t\leq r-1,~and\\
r[\frac{n}{r+2}]+r,~if~ t=r,r+1;
\end{cases}\\
\end{equation*}

$(c)$ for $r\geq2, p\geq2$, $n\geq1$ and $t=\mod(n,2p+r)$
\begin{equation*}
\delta(I\Gamma_{n}^{(p,r)})=
\begin{cases}
2[\frac{n}{r+2p}]+t,~if~t=0,1,~and\\
2[\frac{n}{r+2p}]+2,~if~2\leq t\leq r+2p-1.
\end{cases}\\
\end{equation*}
\end{thm}

\begin{proof}
For convenience,
the vertex with the degree of $\delta(I\Gamma_{n}^{(p,r)})$ in Corollary \ref{cor5.3} is denoted as $I\delta_{n}^{(p,r)}$.
To determine $\delta(I\Gamma_{n}^{(p,r)})$,
we only need to count the number of element $1$s in $\delta_{n}^{(p,r)}$ such that can be changed to $0$.

Let $p,r$ and $n\geq1$.
Then $I\delta_{n}^{(p,r)}$ can be constructed as follows based on Corollary \ref{cor5.3}.
For $1\leq n\leq 2p+r$, let
$I\delta_{1}^{(p,r)}=1,$
$\ldots,$
$I\delta_{r}^{(p,r)}=1^{r},$
$I\delta_{r+1}^{(p,r)}=1^{r}0,$
$\ldots,$
$I\delta_{r+p}^{(p,r)}=1^{r}0^{p},$
$I\delta_{r+p+1}^{(p,r)}=01^{r}0^{p},$
$\ldots,$
$I\delta_{r+2p}^{(p,r)}=0^{p}1^{r}0^{p};$
for $n\geq2p+r,$ let
$I\delta_{n}^{(p,r)}=0^{p}1^{r}0^{p}I\delta_{n-2p-r}^{(p,r)}.$

Since the number of element $1$s in $I\delta_{n}^{(p,r)}$ that can be changed to $0$ according to the value of $r$ and $p$,
we distinguish three cases: $r=1$ and $p\geq1$, $r\geq2$ and $p=1$, $r\geq2$ and $p\geq2$.

First we consider the case $r=1$, $p\geq1$ and $n\geq1$.
Obviously if $r=1$, then all the element $1$s in a vertex of $I\Gamma_{n}^{(p,r)}$ can be changed to $0$.
Since the number of $1$s containing in $I\delta_{n}^{(p,1)}$ is $\lceil \frac{n}{2p+1}\rceil$ for all $p\geq1$ and $n\geq1$,
$\delta(I\Gamma_{n}^{(p,1)})=\lceil \frac{n}{2p+1}\rceil$.
So $(a)$ is proven.

Next we consider the case $r\geq2$, $p=1$ and $n\geq1$.
If $p=1$,
then the neighborhood vertices of every the above vertex $I\delta_{n}^{(p,r)}$ can be obtained by changing every $1$ in this vertex.
It can be calculated that the number of the $1$s in $I\delta_{n}^{(p,r)}$ is
$r[\frac{n}{r+2}]+t$ if $0\leq t\leq r-1$, and
$r[\frac{n}{r+2}]+r$ if $t=r,r+1$.
Hence, $(b)$ holds.

Finally, we consider the case $r\geq2$, $p\geq2$ and $n\geq1$.
The neighborhood vertices set $N(I\delta_{n}^{(p,r)})$ of every the above vertex $I\delta_{n}^{(p,r)}$ can
be obtained by changing the coordinates of $I\delta_{n}^{(p,r)}$ that at the beginning and end of each blocks consisting of $1$s.
Therefore,
the degree of $I\delta_{n}^{(p,r)}$ is shown as in $(c)$.
This completes the proof.
\end{proof}

Recall that if $p=1,r=1$ and $n\geq1$,
then $O\Gamma_{n}^{(p,r)}\cong I\Gamma_{n}^{(p,r)}\cong \Gamma_{n}$ by Proposition \ref{prop2.3}.
It was shown that $\delta(\Gamma_{n})=\lfloor\frac{n+2}{3}\rfloor$ \cite{AKLR},
and this result is contained in Theorem \ref{thm5.4} $(a)$ and $(b)$.

Now we turn to the cube $O\Gamma_{n}^{(p,r)}$.
If $r=1$ or $p=1$,
then $O\Gamma_{n}^{(p,r)}\cong I\Gamma_{n}^{(p,r)}$ by Proposition \ref{prop2.3}.
Therefore, $\delta(O\Gamma_{n}^{(p,r)})=\delta(I\Gamma_{n}^{(p,r)})$ for the case $r=1$ or $p=1$.
Since Theorem \ref{thm5.4} $(a)$ and $(b)$ determine $\delta(I\Gamma_{n}^{(p,r)})$ for the cases $r=1$ and $p=1$ respectively,
we only determine $\delta(O\Gamma_{n}^{(p,r)})$ for the case $r\geq2$ and $p\geq2$.
For convenience,
the vertex with the degree of $\delta(O\Gamma_{n}^{(p,r)})$ in Corollary \ref{cor5.3} is denoted as $O\delta_{n}^{(p,r)}$.
By Corollary \ref{cor5.3},
a vertex with the degree of $\delta(O\Gamma_{n}^{(p,r)})$ can be found as follows.

\begin{lem}\label{lem5.5}
Let $\alpha$ be a vertex of $O\Gamma_{n}^{(p,r)}$,
where $r\geq2$ and $p\geq2$.
Then there are two different form of vertex $O\delta_{n}^{(p,r)}$ can be chosen:

$(i)$ $O\delta_{n}^{(p,r)}=\ldots0^{p-1}10^{2p-1}(10^{p-1})^{r-2}10^{2p-1}10^{p-1}\ldots$, and

$(ii)$ $O\delta_{n}^{(p,r)}=\ldots0^{p-1}10^{2p-2}(10^{2p-2})^{r-2}10^{2p-2}10^{p-1}\ldots$.

\noindent
Specifically, if $p=3$ and $r=3$, or $p=2$ and $r=4$, then both $(i)$ and $(ii)$ can be chosen;
if $p=2$ and $r\geq5$, or $p=3$ and $r\geq4$, or $p\geq4$ and $r\geq3$, then $(ii)$ can be chosen;
if $p=2$ and $2\leq r\leq3$, or $p\geq3$ and $r=2$, then $(i)$ can be chosen.
\end{lem}

\begin{proof}
By Corollary \ref{cor5.3},
It is easy to see that the form of the vertex with the degree of $\delta(O\Gamma_{n}^{(p,r)})$ may be $(i)$ and $(ii)$.
But which form can be chosen depend on the different value of $p,r$ and $n$,
and this can be determined by considering the length of $(i)$ and $(ii)$ such that containing the same element $1$s.
Suppose that both $(i)$ and $(ii)$ contain $r+1$ element $1$s,
that is the degree of both $(i)$ and $(ii)$ are $r+1$.
Then the length of $(i)$ and $(ii)$ are $(r+4)p-1$ and $(2p-1)(r+1)$ respectively.
It is clear that when $p=3$ and $r=3$, or $p=2$ and $r=4$,
$(i)$ and $(ii)$ have the same length;
when $p=2$ and $r\geq5$, or $p=3$ and $r\geq4$, or $p\geq4$ and $r\geq3$, the length of $(ii)$ is bigger than $(i)$;
and when $p=2$ and $2\leq r\leq3$, or $p\geq3$ and $r=2$, the length of $(i)$ is bigger than $(ii)$.
Since the vertices $(i)$ and $(ii)$ have the same degree and the degree of $\delta(O\Gamma_{n}^{(p,r)})$ is considered,
the vertex of the longer length is needed.
The proof is completed.
\end{proof}

\begin{thm}\label{thm5.6}
Let $p\geq2,r\geq2$ and $n\geq1$.
Then $\delta(O\Gamma_{n}^{(p,r)})$ can be determined as follows:

$(a)$ for $r=2$ and $p\geq2$,
\begin{equation*}
\delta(O\Gamma_{n}^{(p,r)})=
\begin{cases}
1,~if~n\leq 2p-1,\\
[\frac{n}{2p}]+1,~if~n\geq 2p;
\end{cases}\\
\end{equation*}

$(b)$ for $r=3, p=2$ and $t=\mod(n,13)$

\begin{equation*}
\delta(O\Gamma_{n}^{(p,r)})=
\begin{cases}
4[\frac{n}{13}],~if~t=0,\\
4[\frac{n}{13}]+k,~if~3k-2\leq t\leq 3k, 1\leq k\leq 4;~and
\end{cases}\\
\end{equation*}

$(c)$ for $r\geq4,p=2$ or $r\geq3, p\geq3$,
$\delta(O\Gamma_{n}^{(p,r)})=\lceil \frac{n}{2p-1}\rceil$.
\end{thm}

\begin{proof}
It is easy to see that the cases $p=2,2\leq r\leq3$ and $p\geq3,r\leq2$ in Lemma \ref{lem5.5} can be recombined into the cases
$p=2,r=3$ and $p\geq2,r=2$.
For the two latter cases,
it is shown in Lemma \ref{lem5.5} that the form of the vertex with the degree of $\delta(O\Gamma_{n}^{(p,r)})$ is $(i)$,
but for $r=2$ and $r=3$,
there are different forms.
For the case $p\geq2$ and $r=2$,
if $n\leq 2p-1$, then let
$O\delta_{1}^{(p,r)}=1,$
$O\delta_{2}^{(p,r)}=01,$
$\ldots$
$O\delta_{p}^{(p,r)}=0^{p-1}1,$
$O\delta_{p+1}^{(p,r)}=0^{p-1}10,$
$\ldots$
$O\delta_{2p-1}^{(p,r)}=0^{p-1}10^{p-1};$
if $2p\leq n \leq4p-1$,
then let
$O\delta_{2p}^{(p,r)}=0^{p-1}10^{p-1}1,$
$O\delta_{2p+1}^{(p,r)}=0^{p-1}10^{p-1}01,$
$\ldots$
$O\delta_{3p-1}^{(p,r)}=0^{p-1}10^{p-1}0^{p-1}1,$
$O\delta_{3p}^{(p,r)}=0^{p-1}10^{p-1}0^{p}1,$
$O\delta_{3p+1}^{(p,r)}=0^{p-1}10^{p-1}0^{p}10,$
$\ldots$
$O\delta_{4p-1}^{(p,r)}=0^{p-1}10^{p-1}0^{p}10^{p-1}$;
if $n\geq 4p$,
then let
$O\delta_{n}^{(p,r)}=0^{p-1}10^{p}\delta_{n-2p}^{(p,r)}.$
The number of $1$s contained in $O\delta_{n}^{(p,r)}$ can be calculate as follows.
If $n\leq2p-1$, then $\delta(O\Gamma_{n}^{(p,r)})=1$;
if $n\geq2p$, then $\delta(O\Gamma_{n}^{(p,r)})=[\frac{n}{2p}]+1$.
So $(a)$ is proved.

For the case $p=2$ and $r=3$,
if $n\leq13$,
then let
$O\delta_{3k+1}^{(p,r)}=(010)^{k}1,$
$O\delta_{3k+2}^{(p,r)}=(010)^{k}10,$
$O\delta_{3k+3}^{(p,r)}=(010)^{k}010,$ where $0\leq k\leq4$,
if $n=13$,
then let
$O\delta_{13}^{(p,r)}=0100010100010$,
and if for $n>13$,
then let
$O\delta_{n}^{(p,r)}=O\delta_{13}^{(p,r)}O\delta_{n-13}^{(p,r)}$.
Let $n=13s+t$ for some $s\geq0$ and $0\leq t\leq12$.
Then $O\delta_{n}^{(p,r)}=(O\delta_{13}^{(p,r)})^{s}O\delta_{t}^{(p,r)}$,
and the number of element $1$s contained in $O\delta_{n}^{(p,r)}$ is as shown in $(b)$.

The cases $p=2,r=4$ and $p=2,r\geq5$ can be combined into one case $p=2, r\geq4$;
the cases $p=3,r=3$, $p=3,r\geq4$ and $p\geq4,r\geq3$ can be combined into one case $p\geq3,r\geq3$.
For both the cases $p=2, r\geq4$ and $p\geq3,r\geq3$,
we can find the vertex with the degree of $\delta(O\Gamma_{n}^{(p,r)})$ with the form $(ii)$ by Lemma \ref{lem5.5},
as shown in the following:
if $n<2p-1$,
then let
$O\delta_{1}^{(p,r)}=1,$
$O\delta_{2}^{(p,r)}=10,\ldots,$
$O\delta_{p}^{(p,r)}=10^{p-1},$
$O\delta_{p+1}^{(p,r)}=010^{p-1},\ldots,$
$O\delta_{2p-1}^{(p,r)}=0^{p-1}10^{p-1};$
if $n\geq2p-1$, then let
$O\delta_{n}^{(p,r)}=0^{p-1}10^{p-1}\delta_{n-2p+1}^{(p,r)}.$
It is clear that the degree of $O\delta_{n}^{(p,r)}$ is $\lceil \frac{n}{2p-1}\rceil$.
So $(c)$ holds.
This completes the proof.
\end{proof}

\smallskip
\section{Concluding remarks}
In this section,
some questions are posed for further investigation.

\begin{pro}
Determine the diameter of $I\Gamma_{n}^{(p,r)}$ for $p\geq2$, $r\geq2p+3$ and $n\geq 2p+3$.
\end{pro}

For the case in the above problem,
only a bound for $diam(I\Gamma_{n}^{(p,r)})$ is given in Theorem \ref{thm4.5},
and it is shown that $diam(I\Gamma_{n}^{(p,r)})$ is bigger than $n$ for some cases.
We think that the fundamental reason for this is that there exist the so called `distance-barrier' in some vertices of
$I\Gamma_{n}^{(p,r)}$.

\begin{pro}\label{pro2}
Determine the radius and center of $I\Gamma_{n}^{(p,r)}$ for $p\geq2$ and $r\geq2$.
\end{pro}

On this problem,
we can prove that $rad(I\Gamma_{n}^{(p,r)})=r\lceil \frac{n}{p+r}\rceil$,
and $0^{n}$ is a center of $I\Gamma_{n}^{(p,r)}$ for $p\geq1,r\geq1$ and $mod(n,p+r)=0$.
For the other case, we think it is a more difficult problem.

In Section 4,
the minimum and maximum degree of $I\Gamma_{n}^{(p,r)}$ and $O\Gamma_{n}^{(p,r)}$ are determined,
respectively.
More generality, we have the following problem.

\begin{pro}
Determine the degree sequence of $I\Gamma_{n}^{(p,r)}$ ($O\Gamma_{n}^{(p,r)}$) for $p\geq1,r\geq1$ and $n\geq1$.
\end{pro}

Connectivity of some Fibonacci like-cubes was studied \cite{AKLR}.
We have the following conjecture.

\begin{con}
The connectivity of $I\Gamma_{n}^{(p,r)}$ and $O\Gamma_{n}^{(p,r)}$ equal to their minimum degree.
\end{con}

It was determined the cubes $I\Gamma_{n}^{(p,r)}$ which are the $Z$-transformation graphs and product graphs in \cite{OZY} and \cite{KGM}, respectively.
We have the following problems.

\begin{pro}
Determine all $O\Gamma_{n}^{(p,r)}$ which can be the
$Z$-transformation graphs of perfect matchings of plane (bipartite) graphs.
\end{pro}

\begin{pro}
Determine $O\Gamma_{n}^{(p,r)}$ which are product graphs.
\end{pro}

\smallskip
\section{Acknowledgment}
This research was funded by projects ZR2018MA010 and ZR2019YQ02 by Shandong Provincial Natural Science Foundation, China.

\end{document}